\documentclass{article}

\newcommand{\Addresses}{{
		\bigskip
		\footnotesize
		
		\textsc{Mathematical Institute, University of Oxford, Oxford, OX2 6GG, UK}\par\nopagebreak
		\textit{E-mail address:} \texttt{gorodetsky@maths.ox.ac.uk}
}}
\author{Ofir Gorodetsky} \title{Smooth integers and de Bruijn's approximation $\Lambda$}
\date{}
\usepackage[margin=1in]{geometry}
\usepackage{amssymb}
\usepackage{amsfonts}
\usepackage{amsthm}
\usepackage{amsmath}
\usepackage{mathtools}
\usepackage{hyperref}

\newtheorem{thm}{Theorem}[section]
\newtheorem{lem}[thm]{Lemma}
\newtheorem{proposition}[thm]{Proposition}
\newtheorem{cor}[thm]{Corollary}

\theoremstyle{definition}

\newtheorem{remark}{Remark}

\newcommand{\CC}{\mathbb{C}}

\newcommand{\ZZ}{\mathbb{Z}}

\newcommand{\RR}{\mathbb{R}}

\mathtoolsset{showonlyrefs}

\numberwithin{equation}{section}

\newcommand{\betasad}{\beta}

\begin{document}
\maketitle
\begin{abstract}
This paper is concerned with the relationship of $y$-smooth integers and de Bruijn's approximation $\Lambda(x,y)$. Under the Riemann hypothesis, Saias proved that the count of $y$-smooth integers up to $x$, $\Psi(x,y)$, is asymptotic to $\Lambda(x,y)$ when $y \ge (\log x)^{2+\varepsilon}$. We extend the range to $y \ge (\log x)^{3/2+\varepsilon}$ by introducing a correction factor that takes into account the contributions of zeta zeros and prime powers. 

We use this correction term to uncover a lower order term in the asymptotics of $\Psi(x,y)/\Lambda(x,y)$. The term relates to the error term in the prime number theorem, and implies that large positive (resp.~negative) values of $\sum_{n \le y} \Lambda(n)-y$ lead to large positive (resp.~negative) values of $\Psi(x,y)-\Lambda(x,y)$, and vice versa. 
Under the Linear Independence hypothesis, we show a Chebyshev's bias in $\Psi(x,y)-\Lambda(x,y)$.
\end{abstract}

\section{Introduction}
A positive integer is called $y$-smooth if each of its prime factors does not exceed $y$. We denote the number of $y$-smooth integers not exceeding $x$ by $\Psi(x,y)$. We assume throughout $x \ge y \ge 2$. Let $\rho\colon [0,\infty) \to (0,\infty)$ be the Dickman function, defined as $\rho(t)=1$ for $t \in [0,1]$ and via the delay differential equation $t \rho'(t) =-\rho(t-1)$ for $t>1$. Dickman \cite{dickman1930} showed that
\begin{equation}\label{eq:dickman} 
	\Psi(x,y) \sim x \rho ( \log x/\log y) \qquad (x \to \infty)
\end{equation}
holds when $y \ge x^{\varepsilon}$. For this reason, it is useful to introduce
\[ u:=\log x/\log y.\]
De Bruijn \cite[Eqs.~(1.3), (4.6)]{debruijn1951} showed that
\begin{equation}\label{eq:err} 
\Psi(x,y)- x\rho(u) \sim (1-\gamma)\frac{x\rho(u-1)}{\log x}>0
\end{equation}
when $x \to \infty$ and $(\log x)/2 >\log y > (\log x)^{5/8}$. Here and later $\gamma$ is the Euler--Mascheroni constant. As we see, there is no arithmetic information in the leading behaviour of the error term $\Psi(x,y) - x \rho(u)$, and in particular it does not oscillate. Moreover, the error term is large: the saving \eqref{eq:err} gives over the main term is merely $\asymp \log (u+1)/\log y$ \cite[p.~56]{debruijn1951}.

This begs the question, what is the correct main term for $\Psi(x,y)$ that leads to a small and arithmetically-rich error term?
De Bruijn \cite[Eq.~(2.9)]{debruijn1951} introduced a refinement of $\rho$, often denoted $\lambda_y$:
\[ \lambda_y(u) := \int_{0}^{\infty} \rho\left(u-\frac{\log t}{\log y}\right)d\left(\frac{\lfloor t \rfloor}{t}\right) = \int_{\RR} \rho(u-v)d\left(\frac{\lfloor y^v \rfloor}{y^v}\right)\]
if $y^u \notin \ZZ$; otherwise $\lambda_y(u)=\lambda_y(u+)$ (one has $\lambda_y(u)= \lambda_y(u-)+O(1/x)$ if $y^u \in \ZZ$ \cite[p.~54]{debruijn1951}). The count $\Psi(x,y)$ should be compared to 
\[ \Lambda(x,y):= x \lambda_y(u).\] 
We refer the reader to de Bruijn's original paper for the motivation for this definition. In particular, $\Lambda$ satisfies the following continuous variant of Buchstab's identity:
\[ \Lambda(x,y)=\Lambda(x,z)-\int_{y}^{z}\Lambda\left(\frac{x}{t},t\right)\frac{dt}{\log t}\]
for $y \le z$, to be compared with 
$\Psi(x,y)=\Psi(x,z)-\sum_{y < p \le z}\Psi(x/p,p)$.
De Bruijn proved \cite[Eq.~(1.4)]{debruijn1951}
\begin{equation}\label{eq:lambdayratio} \Lambda(x,y)=x\rho(u)\left(1+O_{\varepsilon}\left(\frac{\log(u+1)}{\log y}\right)\right)
\end{equation}
holds for $\log y > \sqrt{\log x}$. Saias \cite[Lem.~4]{Saias1989} improved the range to  $y \ge (\log x)^{1+\varepsilon}$. De Bruijn and Saias also provided asymptotic series expansion for $\lambda_y(u)$ in (roughly) powers of $\log (u+1)/\log y$.
Hildebrand and Tenenbaum \cite[Lem.~3.1]{Hildebrand1993} showed  that for $y \ge (\log x)^{1+\varepsilon}$, 
\begin{equation}\label{eq:lambdasize}
	\Lambda(x,y) \asymp_{\varepsilon} x\rho(u)
\end{equation}
for $y \ge (\log x)^{1+\varepsilon}$. Implicit in the proof of Proposition 4.1 of La Bret\`eche and Tenenbaum \cite{LaBreteche2005} is the estimate
\begin{equation}\label{eq:LaBTT}
	\Lambda(x,y) = x \rho(u) K\left( - \frac{\xi(u)}{\log y}\right) \left(1 + O_{\varepsilon}\left( \frac{1}{\log x}\right)\right), \qquad K(t):=\frac{t\zeta(t+1)}{t+1},
\end{equation}
for $y \ge (\log x)^{1+\varepsilon}$ where $\zeta$ is the Riemann zeta function and  $\xi\colon [1,\infty) \to [0,\infty)$ is defined via 
\[e^{\xi(u)}=1+u\xi(u).\] 
We include as an appendix a proof in English of \eqref{eq:LaBTT}. The function $K$ originates in de Bruijn's work \cite[Eq.~(2.8)]{debruijn1951}.
Evidently, $K(0)=1$ and $\lim_{t \to -1^+} K(t)= \infty$. Moreover, $K$ is strictly decreasing in $(-1,0]$  \cite{gorodetsky2022smooth}.

Suppose $\pi(x)=\mathrm{Li}(x)(1+O(\exp(-(\log x)^{a})))$ for some $a \in (0,1)$. Saias \cite[Thm.]{Saias1989}, improving on De Bruijn \cite{debruijn1951}, proved that
\begin{equation}\label{eq:saiaslambdares}
	\Psi(x,y) = \Lambda(x,y)(1 + O_{\varepsilon}(\exp(-(\log y)^{a-\varepsilon})))
\end{equation}
holds in the range $\log y \ge (\log \log x)^{\frac{1}{a}+\varepsilon}$. By the Vinogradov--Korobov zero-free region, we may take $a=3/5$. Saias writes without proof \cite[p.~81]{Saias1989} that under the Riemann hypothesis (RH) his methods give
\begin{equation}\label{eq:saiasrh} \Psi(x,y) = \Lambda(x,y) (1+O_{\varepsilon}(y^{\varepsilon-1/2}\log x))
\end{equation}
in the range $y \ge (\log x)^{2+\varepsilon}$, which recovers a conditional result of Hildebrand \cite{Hildebrand1984}. 
\subsection{\texorpdfstring{$G$}{G}}\label{sec:gdefs}
Define the entire function $I(s)=\int_{0}^{s} \tfrac{e^v-1}{v}dv$. As shown in \cite[Lem.~2.6]{Hildebrand1993}, the Laplace transform of $\rho$ is 
	\begin{equation}\label{eq:hat rho}
		\hat{\rho}(s) := \int_{0}^{\infty} e^{-sv}\rho(v)\, dv = \exp( \gamma + I(-s))
	\end{equation}
for all $s \in \CC$. In \cite{gorodetsky2022smooth} we studied in detail the ratio
\[ G(s,y) := \zeta(s,y) / F(s,y)\]
where
\[ \zeta(s,y):=\prod_{p \le y} (1-p^{-s})^{-1}=\sum_{n \text{ is }y\text{-smooth}} n^{-s} \qquad (\Re s >0)\]
is the partial zeta function and 
\begin{equation}\label{eq:Fdef}
	F(s,y):= \hat{\rho}((s-1)\log y)\zeta(s)(s-1)\log y.
\end{equation}
The function $G(s,y)$ is defined for $\Re s>0$ such that $\zeta(s) \neq 0$. Informally, $G$ carries information about the ratio $\Psi(x,y)/\Lambda(x,y)$, since $s\mapsto \zeta(s,y)/s$ is the Mellin transform of $x\mapsto \Psi(x,y)$ while $s\mapsto F(s,y)/s$ is the Mellin transform of $x\mapsto \Lambda(x,y)$ \cite[p.~54]{debruijn1951}.
As in \cite{gorodetsky2022smooth}, it is essential to write $G$ as $G_1 G_2$ where
\begin{align} 
\log G_1(s,y) &= \sum_{n \le y} \frac{\Lambda(n)}{n^{s}\log n}-(\log (\zeta(s)(s-1))+\log \log y+\gamma+ I((1-s)\log y)),\\
\log G_2(s,y) &= \sum_{k \ge 2} \sum_{y^{1/k} <p \le y} \frac{p^{-ks}}{k}.
\end{align}
We assume $\log \zeta(s)$ is chosen to be real when $s>1$.
\subsection{Main results}
Let $\psi(y)=\sum_{n \le y}\Lambda(n)$ and
\begin{equation}\label{eq:betadef}
\beta:=1-\frac{\xi(u)}{\log y}.
\end{equation}
\begin{thm}\label{thm:psierror}
	Assume RH. Fix $\varepsilon \in (0,1)$. Suppose that $x \ge C_{\varepsilon}$ and  $x^{1-\varepsilon} \ge y \ge (\log x)^{2+\varepsilon}$. Then
	\begin{equation}\label{eq:firstpart}
		\Psi(x,y) = \Lambda(x,y) G(\betasad,y) \left( 1 + O_{\varepsilon}\left( \frac{\log (u+1)}{y \log y} (   |\psi(y)-y| + y^{\frac{1}{2}})\right)\right).
	\end{equation}
\end{thm}
The following theorem gives an asymptotic formula for $\Psi(x,y)$ for $y$ smaller than $(\log x)^{2}$.
\begin{thm}\label{thm:psierror2}
	Assume RH.	Fix $\varepsilon \in (0,1/3)$. Suppose that $x \ge C_{\varepsilon}$ and $(\log x)^{3} \ge y \ge (\log x)^{4/3+\varepsilon}$. Then
	\begin{equation}\label{eq:psig}
		\Psi(x,y) = \Lambda(x,y) G(\betasad,y) \left( 1 + O_{\varepsilon}\left( \frac{(\log y)^3}{y^{\frac{1}{2}}} + \frac{(\log x)^3 (\log y)^3}{y^2}\right)\right).
	\end{equation}
If $y \le (\log x)^{2-\varepsilon}$ then the error term can be improved to $O_{\varepsilon} ((\log x)^3/(y^2 \log y))$.
\end{thm}
Theorems \ref{thm:psierror} and \ref{thm:psierror2}, proved in \S\ref{sec:rherror}, show that
\[ \Psi(x,y) \sim \Lambda(x,y) G(\betasad,y)\]
holds when $y/((\log x)^{3/2}(\log \log x)^{-1/2}) \to \infty$. This range is shown to be optimal in Theorem 2.14 of  \cite{gorodetsky2022smooth}. The same theorem also supplies an alternative proof of Theorem \ref{thm:psierror2} when $y \le (\log x)^{2-\varepsilon}$ (the proof can be adapted to cover $(\log x)^{2-\varepsilon} \le y \le (\log x)^3$ as well).
 
Hildebrand showed that RH is equivalent to $\Psi(x,y) \asymp_{\varepsilon} x\rho(u)$ for $y \ge (\log x)^{2+\varepsilon}$ \cite{Hildebrand1984}. He conjectured that $\Psi(x,y)$ is not of size $\asymp x\rho(u)$ when $y \le (\log x)^{2-\varepsilon}$ \cite{Hildebrand1986}. This was recently confirmed by the author \cite{gorodetsky2022smooth}. This also follows (under RH) from Theorem \ref{thm:psierror2}, since $\Lambda(x,y) \asymp_{\varepsilon} x\rho(u)$ for $y \ge (\log x)^{1+\varepsilon}$ while (under RH) $G(\betasad,y) \to \infty$ when $y \le (\log x)^{2-\varepsilon}$ and $x \to \infty$ (this follows from the estimates for $G$ in \cite{gorodetsky2022smooth}, see \S2).

Theorems \ref{thm:psierror} and \ref{thm:psierror2} and their proofs have their origin in our work in the polynomial setting \cite{gorodetsky}, where $\Psi(x,y)$ corresponds to the number of $m$-smooth polynomials of degree $n$ over a finite field, while $\Lambda(x,y)$ is analogous to the number of $m$-smooth permutations of $S_n$ (multiplied by $q^n/n!$). In that setting, the analogue of $G_1(s,y)$ is identically $1$ (the relevant zeta function has no zeros) which makes the analysis unconditional.
\subsection{Applications: sign changes and biases}
From Theorem \ref{thm:psierror} we deduce in \S\ref{sec:deduccor} the following
\begin{cor}\label{cor:psiover}
	Assume RH.	Fix $\varepsilon \in (0,1)$. Suppose that $x\ge C_{\varepsilon}$ and $x^{1-\varepsilon} \ge y \ge (\log x)^{2+\varepsilon}$. Then
	\begin{align}\label{eq:psiG}
		\Psi(x,y)/\Lambda(x,y) &= 1 +\frac{y^{-\betasad}}{\log y} \bigg(- \sum_{|\rho| \le T} \frac{y^{\rho}}{\rho-\betasad} + \frac{y^{\frac{1}{2}}}{2\betasad-1}+ O_{\varepsilon}\big(\frac{y^{\frac{1}{2}}}{\log y } +\frac{y\log^2(yT)}{T}+ \frac{|\psi(y)-y|+y^{\frac{1}{2}}}{u}\big)\bigg)\\
		&= 1 + \frac{y^{-\betasad}}{\log y} ( (\psi(y)-y)(1+O_{\varepsilon}(u^{-1})) + O_{\varepsilon}(y^{\frac{1}{2}}))\\
		&=1 + O_{\varepsilon}((\log (u+1))(\log x)y^{-\frac{1}{2}})
	\end{align}
holds for $T \ge 4$, where the sum is over zeros of $\zeta$.
\end{cor}
Corollary \ref{cor:psiover} implies that large positive (resp.~negative) values of $\psi(y)-y$ lead to large positive (resp.~negative) values of $\Psi(x,y)  -\Lambda(x,y)$ and vice versa. Large and small values of $\psi(y)-y$ were exhibited by Littlewood \cite[Thm.~15.11]{MV}. Note that Corollary \ref{cor:psiover} sharpens \eqref{eq:saiasrh} if $y \le x^{1-\varepsilon}$.\footnote{For $x \ge y \ge x^{1-\varepsilon}$, de Bruijn proved $\Psi(x,y) = \Lambda(x,y) (1 + O_{\varepsilon}( (\log x)^2 /y^{1/2}))$ under RH \cite[Eq.~(1.3)]{debruijn1951}.} 

Let $\pi(x)$ be the count of primes up to $x$ and $\mathrm{Li}(x)$ be the logarithmic integral. It is known that $\pi(x)-\mathrm{Li}(x)$ is biased towards positive values in the following sense. Assuming RH and the Linear Independence hypothesis (LI) for zeros of $\zeta$, Rubinstein and Sarnak \cite{Rubinstein} showed that the set
\[ \{ x\ge 2 :  \pi(x) > \mathrm{Li}(x)\}\]
has logarithmic density $\approx 0.999997$. This is an Archimedean analogue of the classical Chebyshev's bias on primes in arithmetic progressions.
We use Corollary \ref{cor:psiover} to exhibit a similar bias for smooth integers. Let us fix the value of $\beta=1-\xi(u)/\log y$ to be
\[\betasad=\beta_0\]
where $\beta_0 \in (1/2,1)$.  This amounts to restricting $x$ to be a function $x=x(y)$ of $y$ defined by
\begin{equation}\label{eq:xy} x= \exp\left(\frac{y^{1-\beta_0}-1}{1-\beta_0} \right).
\end{equation}
In particular $y=(\log x)^{1/(1-\beta_0)+o(1)}$. Then Corollary \ref{cor:psiover} shows
\begin{equation}\label{eq:formbias} \frac{\Psi(x(y),y)-\Lambda(x(y),y)}{\Lambda(x(y),y)} y^{\betasad_0-\frac{1}{2}}\log y = -\sum_{|\rho| \le T} \frac{y^{\rho-\frac{1}{2}}}{\rho-\beta_0} + \frac{1}{2\beta_0-1} + O_{\beta_0}\left( \frac{y^{\frac{1}{2}}\log^2 (yT)}{T}+\frac{1}{\log y}\right).
\end{equation}
Applying the formalism of Akbary, Ng and Shahabi \cite{Akbary} to the right-hand side of \eqref{eq:formbias} we deduce immediately
\begin{cor}\label{cor:bias}
	Assume RH. Assume LI for $\zeta$. Fix $\beta_0 \in (1/2,1)$ and let $x$ be a function of $y$ defined as in \eqref{eq:xy}. Then the set
	\[ \{ y \ge 2: \Psi(x(y),y) > \Lambda(x(y),x) \}\]
	has logarithmic density greater than $1/2$, and the left-hand side of \eqref{eq:formbias} has a limiting distribution in logarithmic sense.
\end{cor}
In the same way that Chebyshev's bias for primes relates to the contribution of prime squares, this is also the case for smooth integers. Writing $G$ as $G_1 G_2$ as in \S\ref{sec:gdefs}, $G_2$ captures the contribution of proper powers of primes. When $\beta_0 \in (1/2,1)$, the only significant term in $G_2(\beta_0,y)$ is $k=2$, which corresponds to squares of primes. The squares lead to the term $y^{1/2}/(2\beta_0-1)$ in \eqref{eq:formbias} which creates the bias. 
\begin{remark}
	Consider the arithmetic function $\alpha_y(n)$ defined implicitly via
	\[ \sum_{n\ge 1} \frac{\alpha_y(n)}{n^s} = \exp\bigg(\sum_{m \le y} \frac{\Lambda(m)}{ \log m}\frac{1}{m^s}\bigg).\]
	This function is supported on $y$-smooth numbers and coincides with the indicator of $y$-smooth numbers on squarefree integers. Working with the summatory function of $\alpha_y$ instead of $\Psi(x,y)$, the bias discussed above disappears. This is because, modifying the proof of Theorem \ref{thm:psierror}, one finds that
	\[ \sum_{n \le y} \alpha_y(n) = \Lambda(x,y) G_1(\beta,y)\left( 1 + O_{\varepsilon}\left( \frac{\log (u+1)}{y \log y} ( |\psi(y)-y|+y^{\frac{1}{2}})\right)\right)\]
	holds in $x^{1-\varepsilon} \ge y \ge (\log x)^{2+\varepsilon}$, meaning the bias-causing factor $G_2(\beta,y)$ does not arise. This is analogous to how the indicator function of primes is biased, while $\Lambda(n)/\log n$ is not.
\end{remark}
\begin{remark}
It is interesting to see if one can formulate and prove variants of Corollaries \ref{cor:psiover} and \ref{cor:bias}  in the range $y \le (\log x)^{1-\varepsilon}$. In this range, an accurate main term for $\Psi(x,y)$ was established in \cite{LaBreteche}.
\end{remark}
\subsection{Strategy behind Theorems \ref{thm:psierror} and \ref{thm:psierror2}}
We write $\Psi(x,y)$ as a Perron integral, at least for non-integer $x$:
\[ \Psi(x,y) = \frac{1}{2\pi i} \int_{(\sigma)} \zeta(s,y) \frac{x^s}{s} ds \]
where $\sigma$ can be any positive real. For non-integer $x$ we also have
\begin{equation}\label{eq:perronforlambda}
	\Lambda(x,y)=\frac{1 }{2\pi i}\int_{(\sigma)} F(s,y) \frac{x^s}{s}ds
\end{equation}
whenever $\sigma>\varepsilon$ and $y \ge C_{\varepsilon}$. Indeed, the Laplace inversion formula expresses $\Lambda(x,y)$ as
\begin{equation}\label{eq:perronforlambda2}
	\Lambda(x,y)=x\lambda_y(u)=\frac{x}{2\pi i}\int_{(c)}\hat{\lambda}_y(s)e^{us}ds=\frac{1 }{2\pi i}\int_{(1+\frac{c}{\log y})}( \hat{\lambda}_y((s-1)\log y) \log y) x^s ds
\end{equation}
for any $c$ such that 
\begin{equation}\label{eq:laplacelambday} \hat{\lambda}_y(s):=\int_{0}^{\infty} e^{-sv} \lambda_y(v)dv,
\end{equation}
converges absolutely for $\Re s \ge c$. In particular, we may take $c>-(\log y)/(1+\varepsilon)$ if we assume $y\ge C_{\varepsilon}$, as Saias showed, see Corollary \ref{cor:lambdaconv}. 
As shown by de Bruijn \cite[Eq.~(2.6)]{debruijn1951} (cf.~\cite[Lem.~6]{Saias1989}),
\[ \hat{\lambda}_y(s) = \hat{\rho}(s) K(s/\log y).\]
By definition of $F$, \eqref{eq:Fdef}, we can rewrite \eqref{eq:perronforlambda2} as \eqref{eq:perronforlambda}.
As Saias does, we choose to work with $\sigma=\betasad$, which is essentially a saddle point for $F(s,y)x^s$. If $x \ge y \ge (\log x)^{1+\varepsilon}$ and $x \ge C_{\varepsilon}$ then Lemma \ref{lem:xilem} implies
\[ \betasad \ge c_{\varepsilon}>0.\]
Saias proved \eqref{eq:saiaslambdares} by showing that $\zeta(s,y)$ and $F(s,y)$ are close and so if we subtract
\[ \Psi(x,y) -\Lambda(x,y)  = \frac{1}{2\pi i} \int_{(\betasad)} (\zeta(s,y)-F(s,y)) \frac{x^s}{s}ds\]
then we can bound the integral by using pointwise bounds for the integrand. 
Instead of subtracting $\Lambda(x,y)$, we subtract $\Lambda(x,y)$ times $G(\betasad,y)$, which leads to
\begin{equation}\label{eq:secondformula}
		\Psi(x,y) =\Lambda(x,y) G(\betasad,y) \left(1 +  \frac{\Lambda(x,y)^{-1}}{2\pi i}  \int_{(\betasad)} \frac{G(s,y) - G(\betasad,y)}{G(\betasad,y)}F(s,y) \frac{x^s}{s} ds\right).
	\end{equation}
We want to bound the integral in \eqref{eq:secondformula}. The proof of Theorem \ref{thm:psierror} considers separately the range
\begin{equation}\label{eq:usmall}
u \ge (\log y) ( \log \log y)^3
\end{equation}
and its complement. When $u$ satisfies \eqref{eq:usmall}, then in \eqref{eq:secondformula} one needs only small values of $\Re s$ to estimate the integral ($|\Re s| \le 1/\log y$) with arbitrary power saving in $y$. This is an unconditional observation established in Proposition \ref{prop:medu}. However, for smaller $u$, one needs $|\Re s|$ going up to a power of $y$ if one desires power saving in $y$, which makes the proof more involved.

In our proofs, RH is only invoked at the very end to estimate $G_1$ and its derivatives. For instance, in the range where \eqref{eq:usmall} and $y \ge (\log x)^{2+\varepsilon}$ hold, we prove in \eqref{eq:psiuncon} the \textit{unconditional} estimate
	\begin{equation}\label{eq:psiuncon2}
		\Psi(x,y) = \Lambda(x,y) G(\betasad,y) \left( 1+O_{\varepsilon}\left( \frac{\max_{|v|\le 1} |G'(\betasad+iv,y)|}{G(\betasad,y) \log x} + \frac{ \max_{|v|\le 1}|G''(\betasad+iv,y)|}{G(\betasad,y) (\log x)(\log y)} + \frac{1}{y}\right)\right).
	\end{equation}
See \eqref{eq:smallutriangle} for a similar estimate for $u \le  (\log y)(\log \log y)^3$. In particular, our proofs are easily modified to recover \eqref{eq:saiaslambdares}.
\subsection*{Conventions}
The letters $C,c$ denote absolute positive constants that may change between different occurrences. We denote by $C_{\varepsilon},c_{\varepsilon}$ positive constants depending only on $\varepsilon$, which may also change between different occurrences. 
The notation $A \ll B$ means $|A| \le C B$ for some absolute constant $C$, and $A\ll_{\varepsilon} B$ means $|A| \le C_{\varepsilon} B$. We write $A \asymp B$ to mean $C_1 B \le A \le C_2 B$ for some absolute positive constants $C_i$, and $A \asymp_{\varepsilon} B$ means $C_i$ may depend on $\varepsilon$. The letter $\rho$ will always indicate a non-trivial zero of $\zeta$. When we differentiate a bivariate function, we always do so with respect to the first variable. We set 
\[ L(y):=\exp((\log y)^{\frac{3}{5}}(\log \log y)^{-\frac{1}{5}}).\]
\section{Preliminaries}\label{sec:prelim}
\subsection{Standard lemmas}
Recall $\beta$ was defined in \eqref{eq:betadef}.
\begin{lem}\label{lem:xilem}\cite[Lem.~1]{HildebrandTenenbaum1986}
	For $u \ge 3$ we have $\xi(u) = \log u + \log \log u + O( (\log \log u) / \log u)$. In particular,
	\begin{equation}\label{eq:ybetasize}
		y^{1-\beta} \asymp u \log(u+1), \qquad u \ge 1. 
	\end{equation}
\end{lem}
\begin{lem}\cite{debruijn19512}\label{lem:rho size}
For $u \ge 1$ we have $\rho(u) \asymp e^{-u\xi +I(\xi)} u^{-1/2} = x^{\beta-1} e^{I(\xi)}u^{-1/2}$.
\end{lem}
In the next lemmas we write $s \in \CC$ as $s=\sigma + it$.
\begin{lem}\cite[Cor.~10.5]{MV}\label{lem:functional}
	For $|\sigma| \le A$ and $|t| \ge 1$, $|\zeta(s)| \asymp_A (|t|+4)^{1/2-\sigma}|\zeta(1-s)|$.
\end{lem}
\begin{lem}\cite[Cor.~1.17]{MV}\label{lem:convex}
	Fix $\varepsilon>0$. For $\sigma \in [\varepsilon,2]$ and $|t| \ge 1$ we have
	\[\zeta(s) \ll_{\varepsilon} (1+(|t|+4)^{1-\sigma})\min\left\{\frac{1}{|\sigma-1|},\log (|t|+4)\right\}.\]
\end{lem}
\begin{lem}\cite[Thm.~7.2(A)]{Titchmarsh1986}\label{lem:second}
	We have, for $\sigma \in [1/2,2]$ and $T \ge 2$,
	\[ \int_{1}^{T} |\zeta(\sigma+it)|^2 dt \ll T \min\left\{ \log T, \frac{1}{\sigma-\frac{1}{2}}\right\}.\]
\end{lem}
\begin{lem}\cite[Lem.~2.7]{Hildebrand1993}\label{lem:i bounds}
	The following bounds hold for $s=-\xi(u)+it$:
	\begin{equation}
		\hat{\rho}(s) =e^{\gamma+I(-s)}= \begin{cases} O\left(\exp\left(I(\xi)-\frac{t^2u}{2\pi^2}\right)\right) & \mbox{if }|t| \le \pi,\\
			O\left(\exp\left(I(\xi)-\frac{u}{\pi^2+\xi^2}\right)\right) & \mbox{if }|t| \ge \pi,\\
			\frac{1}{s} + O\left( \frac{1+u \xi }{|s|^2} \right) & \mbox{if }1+u\xi=O(|t|).\end{cases}
	\end{equation}
\end{lem}
The third case of Lemma \ref{lem:i bounds} is usually stated in the range $1+u\xi \le |t|$, but the same proof works for $1+u\xi = O(|t|)$. Since $1+u\xi=e^{\xi}$, the third case can also be written as
\begin{equation}\label{eq:third} s\hat{\rho}(s) = 1+O( e^{-\sigma}/|t|)
\end{equation}
for $s=\sigma+it$, assuming $\sigma <0$ and $e^{-\sigma} =O(|t|)$. The following lemma is a variant of \cite[Lem.~8]{HildebrandTenenbaum1986}, proved in the same way. 
\begin{lem}\cite{HildebrandTenenbaum1986}\label{lem:zetaratio}
	Fix $\varepsilon>0$. Suppose $x \ge y \ge (\log x)^{1+\varepsilon}$ and $x \ge C_{\varepsilon}$. 	For $|t| \le 1/\log y$,
	\[ \left|\frac{\zeta(\betasad+it,y)}{\zeta(\betasad,y)}\right| \le \exp(-ct^2 (\log x )(\log y) ).\]
	For $1/\log y \le |t| \le \exp((\log y)^{3/2-\varepsilon})$,
	\begin{equation}\label{eq:zetar}
		\frac{\zeta(\betasad+it,y)}{\zeta(\betasad,y)} \ll_{\varepsilon} \exp\left(-\frac{c ut^2}{(1-\betasad)^2+t^2}\right).
	\end{equation}
\end{lem}
\subsection{More on \texorpdfstring{$G$}{G}}\label{sec:deduccor}
\begin{lem}\label{lem:logg1formula}\cite{gorodetsky2022smooth}
Fix $0 \le i \le 4$. Let $y \ge 4$. Let $s \in \CC$ with $\Re s \in [0,1]$ and the property that
\begin{equation}\label{eq:d}
\min_{\zeta(\rho)=0,\, t \ge 0} |\rho-s-t| \gg 1.
\end{equation}
Then for $T \ge 3+|\Im s|$  we have
	\begin{equation}
		(\log G_1)^{(i)}(s,y)= -\sum_{|\Im (\rho-s)| \le T}\frac{d^{i}}{d s^{i}} \int_{0}^{\infty} \frac{y^{\rho-s-t}}{\rho-s-t}dt+O\left((\log y)^{i} y^{-\Re s}+ \frac{\log^2 (yT)(\log y)^{i-1} }{T}  y^{1-\Re s}\right).
\end{equation}
\end{lem}
\begin{cor}\label{cor:logg1size}
Fix $0 \le i \le 4$. Let $y \ge 4$. Let $s \in \CC$ with $\Re s \in [0,1]$. If  $|\Im s| \le 1$ we have
$(\log G_1)^{(i)}(s,y) \ll L(y)^{-c} y^{1-\Re s}$ unconditionally. Under RH, if $T \ge 4$ and  $|\Im s| \le 1$ then
		\begin{equation}\label{eq:gen}
			\begin{split}
				(\log G_1)^{(i)}(s,y) &=(-\log y)^{i-1}y^{- s}\bigg( \sum_{|\Im (\rho-s)| \le T} \frac{y^{\rho}}{\rho-s} + O \bigg(  \frac{y^{\frac{1}{2}}}{\log y}+ \frac{y\log^2(yT)}{T}\bigg) \bigg)\\
				&=(-1)^i(\log y)^{i-1} y^{- s} (\psi(y)-y+ O(y^{\frac{1}{2}})) \ll y^{\frac{1}{2}-\Re s}(\log y)^{i+1}.
			\end{split}
		\end{equation}
Under RH, if $T \ge 4$, $\Re s \in [3/4,1]$ and $|\Im s| \le y^{9/10}$ then
\begin{equation}\label{eq:gen2}
		(\log G_1)^{(i)}(s,y) = (-1)^i(\log y)^{i-1} y^{- s} (\psi(y)-y+ O(y^{\frac{1}{2}}\log^2 (|\Im s|+2))) \ll y^{\frac{1}{2}-\Re s}(\log y)^{i+1}.
\end{equation}
\end{cor}
\begin{proof}
If $|\Im s| \le 1$ then \eqref{eq:d} holds. It is easily seen that, for any zero $\rho$ of $\zeta$,
	\begin{equation}\label{eq:xrhoideriv}
			\frac{d^{i}}{ds^{i}}\int_{0}^{\infty} \frac{y^{\rho-s-t}}{\rho-s-t}dt =- \frac{(-\log y)^{i-1}y^{\rho-s}}{\rho-s} \left( 1  +O\left( \frac{1}{\min_{t \ge 0} |\rho-s-t|\log y}\right)\right)
	\end{equation}
if \eqref{eq:d} holds. We apply Lemma \ref{lem:logg1formula} with $T=L(y)^c$ and use the Vinogradov--Korobov zero-free region and \eqref{eq:xrhoideriv} to simplify. Now assume RH, i.e.~$|y^{\rho}|=y^{1/2}$. We demonstrate \eqref{eq:gen}, and \eqref{eq:gen2} is proved along similar lines. We apply Lemma \ref{lem:logg1formula} with $T\ge 4$ and simplify it using \eqref{eq:xrhoideriv}. We bound the resulting error using the facts $\min_{t \ge 0}|\rho-s-t| \asymp |\rho-s|$ and  $\sum_{\rho} 1/|\rho-s|^2 \ll 1$ for $|s|\le 2$, since there are $\ll \log T$ zeros of $\zeta$ between height $T$ and $T+1$ \cite[Thm.~10.13]{MV}. This gives the first equality in \eqref{eq:gen}. The second equality in \eqref{eq:gen} follows by taking $T=y$, recalling the classical estimate \begin{equation}\label{eq:sumz}
\psi(y)-y = -\sum_{|\rho| \le y} \frac{y^{\rho}}{\rho} + O(\log^2 y)
\end{equation}
given in \cite[Thm.~12.5]{MV} (it also follows from Lemma \ref{lem:logg1formula} with $(i,s,T)=(1,0,y)$), and the bound $\sum_{\rho} 1/(|\rho-s||\rho|) \ll 1$. The last inequality in \eqref{eq:gen} is von Koch's bound  $\psi(y)-y=O(y^{1/2}\log^2 y)$ \cite{vonKoch}. 
\end{proof}
We turn to $G_2$. By the non-negativity of the coefficients of $\log G_2$, for $i \ge 0$ and $\Re s>0$ we have 
\begin{equation}\label{eq:pos}
	|(\log G_{2})^{(i)}(s,y)| \le (-1)^i\log G_{2}^{(i)}(\Re s ,y ).
\end{equation}
\begin{lem}\cite{gorodetsky2022smooth}\label{lem:logGderivaccurate2}
	Fix $\varepsilon>0$ and $0 \le i \le 4$. For $y \ge 2$ and $1 \ge s \ge \varepsilon$,	
	\begin{align}	(\log G_2)^{(i)}(s,y) &=( 1+O_{\varepsilon}( L(y)^{-c})) \frac{(-2)^i}{2}\int_{y^{1/2}}^{y}(\log t)^{i-1} t^{-2s} dt \asymp_{\varepsilon}  \frac{(-\log y)^{i}y^{\max\{1-2s,\frac{1}{2}-s\}}}{\max\{1,|s-1/2|\log y\}}.
	\end{align}
\end{lem}
Corollary \ref{cor:logg1size} and Lemma \ref{lem:logGderivaccurate2}, applied with $i=0$, imply the following
\begin{lem}\label{lem:sizeg}
	Assume RH. Fix $\varepsilon>0$. If $1 \ge s \ge 1/2+\varepsilon$ and $T \ge 4$ then
	\begin{align}
		G(s,y) &= 1 +\frac{y^{-s}}{\log y} \bigg(- \sum_{|\rho| \le T} \frac{y^{\rho}}{\rho-s} + \frac{y^{\frac{1}{2}}}{2s-1}+ O_{\varepsilon}\bigg(\frac{y^{\frac{1}{2}}}{\log y} + \frac{y \log^2( yT)}{T}\bigg)\bigg)\\
		&=1 + \frac{y^{-s}}{\log y} ( \psi(y)-y + O_{\varepsilon}(y^{\frac{1}{2}}))=1 + O_{\varepsilon}( y^{\frac{1}{2}-s} \log y).
\end{align}
\end{lem}
Corollary \ref{cor:psiover} follows from Theorem \ref{thm:psierror} by simplifying $G(\beta,y)$ using Lemma \ref{lem:sizeg} and \eqref{eq:ybetasize}.
\section{Truncation estimates for \texorpdfstring{$\Psi$}{Psi} and \texorpdfstring{$\Lambda$}{Lambda}}\label{sec:trunc}
The purpose of this section is to prove the following two propositions.
\begin{proposition}[Medium $u$]\label{prop:medu}
Suppose $x \ge y \ge 2$ satisfy \[u \ge (\log y) ( \log \log y)^3.\]
Fix $\varepsilon>0$. Suppose $y \ge (\log x)^{1+\varepsilon}$ and $x \ge C_{\varepsilon}$. Then
\begin{align}
\label{eq:psimedium}\Psi(x,y) &= \frac{1}{2\pi i} \int_{\betasad-\frac{i}{\log y}}^{\betasad+\frac{i}{\log y}} \zeta(s,y) \frac{x^s}{s}ds + O_{\varepsilon}\left( \frac{\Psi(x,y)+x\rho(u) G(\betasad,y)}{\exp(c_{\varepsilon}\min\{u/\log^2(u+1),(\log y)^{4/3}\})}\right),\\
\label{eq:lambmed}\Lambda(x,y) &= \frac{1}{2\pi i} \int_{\betasad-\frac{i}{\log y}}^{\betasad+\frac{i}{\log y}} F(s,y) \frac{x^s}{s}ds + O_{\varepsilon}\left(\frac{x\rho(u)}{\exp( c u/\log^2(u+1))}\right).
\end{align}
\end{proposition}

\begin{proposition}[Small $u$]\label{prop:small}
Suppose $x \ge y \ge 2$ satisfy
\[u \le (\log y) ( \log \log y)^3.\]
Suppose $x \ge C$ and let $T \in [(\log x)^5,x\rho(u)]$.  Then
\begin{align}
	\Psi(x,y) &= \frac{1}{2\pi i} \int_{\betasad-i T}^{\betasad+iT} \zeta(s,y) \frac{x^s}{s}ds + O\left( \frac{\Psi(x,y)+x\rho(u) G(\betasad,y)}{T^{4/5}}\right),\\
	\Lambda(x,y) &= \frac{1}{2\pi i} \int_{\betasad-iT}^{\betasad+iT} F(s,y) \frac{x^s}{s}ds + O\left(\frac{x\rho(u)}{T^{4/5}}\right).
\end{align}
\end{proposition}
\subsection{Preparation}
\begin{lem} \label{lem:zeta}
Fix $\varepsilon \in (0,1)$. For $\sigma \in [\varepsilon,1]$ and $x \ge T \ge 2$ we have
	\begin{align}\label{eq:2ndbnd}
	\frac{1}{2\pi i}\int_{\sigma+it: \, |t|>T} \zeta(s) \frac{x^s}{s} ds &\ll_{\varepsilon} \frac{x^{\sigma}}{T^{\sigma}}\log T + \log x.
\end{align}
\end{lem}
The integral should be understood in principal value sense. Lemma \ref{lem:zeta} makes more precise a computation done in p.~96 of Saias' paper \cite{Saias1989} (cf.~\cite[p.~537]{Tenenbaum2015}), which is not stated for general $T$ and $\sigma$ but contains the same ideas. 
\begin{proof}
By \cite[Thm.~4.11]{Titchmarsh1986}, for every $r>0$ we have
\[ \zeta(s) =\sum_{n \le r} n^{-s} - \frac{r^{1-s}}{1-s} + O_{\varepsilon}(r^{-\Re s})\]
as long as $s \neq 1$, $\Re s \ge \varepsilon$ and $|\Im s|\le 2r$. Suppose $s=\sigma+it$ with $|t| \ge 1$. We apply this estimate with $r=|t|$, obtaining
\begin{equation}\label{eq:titch} \zeta(s) =\sum_{n \le |t|} n^{-s} - \frac{|t|^{1-s}}{1-s} + O_{\varepsilon}(|t|^{-\sigma})= \sum_{n \le |t|} n^{-s}  + O_{\varepsilon}(|t|^{-\sigma}).
\end{equation}
We now plug \eqref{eq:titch} in the left-hand side of \eqref{eq:2ndbnd}. The contribution of the error term to the integral is acceptable:
\[ \int_{\sigma+it: \, |t|>T} O(|t|^{-\sigma}) \frac{x^s}{s} ds \ll x^{\sigma}\int_{T}^{\infty} |t|^{-\sigma-1}dt \ll_{\varepsilon} \frac{x^{\sigma}}{T^{\sigma}}.\]
The contribution of $n^{-s} \mathbf{1}_{n\le|t|}$ in \eqref{eq:titch} to the left-hand side of \eqref{eq:2ndbnd} is
\begin{equation}\label{eq:n} \frac{1}{2\pi i} \int_{\sigma+it: |t|>\max\{n,T\}} n^{-s} \frac{x^s}{s}ds.
\end{equation}
Since
\[ \frac{1}{2\pi i} \int_{\sigma+it:\, |t| \le S} n^{-s} \frac{x^s}{s}ds =\mathbf{1}_{x>n} + \frac{\mathbf{1}_{x=n}}{2} +O\left( \frac{(x/n)^{\sigma}}{1+S|\log(x/n)|}\right), \qquad S \ge 1,\]
by the truncated Perron's formula \cite[p.~435]{Hildebrand1993}, and \[ \frac{1}{2\pi i} \int_{(\sigma)} n^{-s} \frac{x^s}{s}ds =\mathbf{1}_{x>n} + \frac{\mathbf{1}_{x=n}}{2}  \]
by Perron's formula, it follows that the integral in \eqref{eq:n} is bounded by
\[ \ll \frac{(x/n)^{\sigma}}{1+\max\{n,T\}|\log(x/n)|}\]
and so the total contribution of the $n$-sum in \eqref{eq:titch} to the left-hand side of \eqref{eq:2ndbnd} is
\begin{equation}\label{eq:newnsum}
 \ll x^{\sigma}\sum_{n \ge 1} \frac{n^{-\sigma}}{1+\max\{n,T\}|\log(x/n)|}.
\end{equation}
It remains to estimate \eqref{eq:newnsum}, which we do according to the size of $n$. The contribution of $n \ge 2x$ is 
\[ \ll x^{\sigma} \sum_{n \ge 2x} n^{-\sigma-1} \ll_{\varepsilon} 1.\]
The contribution of $n \in (x/2,2x)$ can be bounded by considering separately the $n$ closest to $x$, and partitioning the rest of the $n$s according to the value of $k\ge 0$ for which $|\log(x/n)| \in [2^{-k},2^{1-k})$:
\[ \ll x^{\sigma}\sum_{n \in (x/2,2x)} \frac{n^{-\sigma}}{1+x|\log (x/n)|}\ll 1+ \sum_{k \ge 0:\, 2^k \le 2x} \frac{x}{2^k} \frac{1}{1+x/2^k} \ll \log x.\]
The contribution of $n \le T/2$ is 
\[ \ll \frac{x^{\sigma}}{T} \sum_{n \le T/2}n^{-\sigma} \ll \frac{x^{\sigma}}{T^{\sigma}} \log T.\]
Finally, the contribution of $T/2<n\le x/2$ is
\[ \ll x^{\sigma} \sum_{n>T/2} n^{-1-\sigma} \ll_{\varepsilon} \frac{x^{\sigma}}{T^{\sigma}},\]
acceptable as well.
\end{proof}
\begin{cor} \label{cor:lambdatrunc}
Fix $\varepsilon \in (0,1)$. Suppose $x \ge y \ge C_{\varepsilon}$.  For $\sigma \in [\varepsilon,1]$ and $x \ge T \ge \max\{2,y^{1-\sigma}/\log y\}$ we have
\[ \Lambda(x,y) = \frac{1}{2\pi i}\int_{\sigma-iT}^{\sigma+iT} F(s,y) \frac{x^s}{s} ds+ O_{\varepsilon}\left( \frac{x^{\sigma}}{T^{\sigma}}\log T+\log x+ x^{\sigma} \frac{y^{1-\sigma}}{\log y}\frac{\log^{1/2} T}{T^{\min\{1,1/2+\sigma\}}} \right).\]
\end{cor}
Corollary \ref{cor:lambdatrunc} rests on Lemma \ref{lem:zeta}, and makes more precise Proposition 2 of Saias \cite{Saias1989}. 
\begin{proof}
Our starting point is the identity \eqref{eq:perronforlambda}. (If $x \in \ZZ$ it still holds with an error term of $O(1)$, since the integral converges to the average $(\Lambda(x+,y)+\Lambda(x-,y))/2 = \Lambda(x,y)+O(1)$.)
From that identity it follows that our task is equivalent to upper bounding
\[ \left|\int_{\sigma+it:\, |t|>T} F(s,y) \frac{x^s}{s} ds\right|.\]
Recall $F(s,y) =  \hat{\rho}((s-1)\log y)\zeta(s)(s-1)\log y$.
By \eqref{eq:third} with $(s-1)\log y$ instead of $s$ we find
\[ F(s,y)=\zeta(s)\left(1+O\left(\frac{y^{1-\sigma}}{|t|\log y}\right)\right)\]
if $y^{1-\sigma} = O(|t| \log y)$, which holds by our assumptions on $T$. By the triangle inequality,
\begin{equation}\label{eq:lambda2int} 
\left|\int_{\sigma+it:\, |t|>T} F(s,y) \frac{x^s}{s} ds\right| \ll \left|\int_{\sigma+it:\, |t|>T} \frac{\zeta(s)}{s}x^s ds\right| + x^{\sigma}\frac{y^{1-\sigma}}{\log y} \int_{\sigma+it:\,|t|>T} \frac{|\zeta(s)|}{|t|^2} |ds| .
\end{equation}
The first integral in the right-hand side of \eqref{eq:lambda2int} is estimated in Lemma \ref{lem:zeta}.
To bound the second integral we apply the second moment estimate for $\zeta$ given in Lemma \ref{lem:second}. We first suppose that $\sigma \ge 1/2$. Using Cauchy--Schwarz, the second integral in the right-hand side of \eqref{eq:lambda2int} is at most
\begin{equation}
\int_{\sigma+it:\,|t|>T} \frac{|\zeta(s)|}{|t|^2} |ds| \ll \sum_{2^k \ge T/2}4^{-k} \int_{2^k}^{2^{k+1}} |\zeta(\sigma+it)|dt\ll \sum_{2^k \ge T/2}2^{-k} k^{1/2} \ll \frac{\log^{1/2} T}{T}.
\end{equation}
Multiplying this by the prefactor $x^{\sigma}y^{1-\sigma}/\log y$, we see that this is acceptable. If $\varepsilon \le \sigma \le 1/2$ we use Lemma \ref{lem:functional}. We obtain that the second integral in the right-hand side of \eqref{eq:lambda2int} is at most
\begin{equation}
	\int_{\sigma+it:\,|t|>T} \frac{|\zeta(s)|}{|t|^2} |ds| \ll \int_{1-\sigma+it:\,|t|>T} \frac{|\zeta(s)|}{|t|^{2+\sigma-1/2}} |ds|  \ll \sum_{2^k \ge T/2} 2^{-k(\sigma+1/2)}k^{1/2} \ll \frac{\log^{1/2} T}{T^{\frac{1}{2}+\sigma}},
\end{equation}
concluding the proof. 
\end{proof}
Let $\alpha=\alpha(x,y)$ be the saddle point associated with $y$-smooth numbers up to $x$ \cite{HildebrandTenenbaum1986}, that is, the  minimizer of the convex function $s\mapsto x^s \zeta(s,y)$ ($s>0$).
\begin{lem}\label{lem:psitrunc}
	For $\sigma\in (0,1]$, $x\ge y \ge C$ and  $T \ge 2$ we have
	\begin{equation}\label{eq:psiperron2} \Psi(x,y)=\frac{1}{2\pi i}\int_{\sigma-iT}^{\sigma+iT} \zeta(s,y)\frac{x^s}{s}ds+O\left( \frac{x^{\sigma}\zeta(\sigma,y)}{T} + \frac{\Psi(x,y) \log T}{T^{\alpha}}+1\right).
	\end{equation}
\end{lem}
Our proof makes more precise a similar estimate appearing in Saias \cite[p.~98]{Saias1989}, which does not allow general $y$ and $T$ but contains the main ideas.
\begin{proof}
	The truncated Perron's formula \cite[p.~435]{Hildebrand1993} bounds the error in \eqref{eq:psiperron2} by
	\[ \ll x^{\sigma}\sum_{\substack{n\ge 1\\ n\text{ is }y\text{-smooth}}}\frac{1}{n^{\sigma}(1+T|\log(x/n)|)}.\]
	The contribution of the terms with $|\log(x/n)|\ge 1$ is
	\[ \ll \frac{x^{\sigma}}{T} \sum_{\substack{n\ge 1\\n\text{ is } y\text{-smooth}}} \frac{1}{n^{\sigma}} = \frac{x^{\sigma}\zeta(\sigma,y)}{T}.\]
	We now study the terms with $|\log(x/n)|<1$. These contribute
	\begin{equation}\label{eq:nclosex}\ll \sum_{\substack{e^{-1}x<n<ex\\n\text{ is }y\text{-smooth}}} \frac{1}{1+T|\log (x/n)|}.
	\end{equation}
	The subset of terms with $|\log(x/n)| \le 1/T$ contributes to \eqref{eq:nclosex}
	\begin{equation}\label{eq:xnclose} \ll \sum_{\substack{|n-x| \le Cx/T \\ n\, y\text{-smooth}}} 1 \ll \Psi\left(x+ \frac{Cx}{T},y\right)-\Psi\left(x-\frac{Cx}{T},y\right).
	\end{equation}
	The contribution of the rest of the terms to to \eqref{eq:nclosex}, namely, those terms with $1/T<|\log(x/n)| <1$, can be dyadically dissected to terms with $|\log(x/n)| \in [2^{-k},2^{1-k})$ for each integer $k\ge 1$ such that $2^k<2T$ holds. Their total contribution is
	\begin{equation}\label{eq:dyadicpsi} \ll \frac{1}{T}\sum_{1 \le k \le \log_2 T + 1} 2^k\left(\Psi\left(x+ \frac{Cx}{2^k},y\right)-\Psi\left(x-\frac{Cx}{2^k},y\right)\right),
	\end{equation}
	where $\log_2$ is the base-2 logarithm. (We interpret $\Psi(a,y)$ for negative $a$ as equal to $0$.) Note that the sum in \eqref{eq:dyadicpsi} dominates the right-hand side of \eqref{eq:xnclose}.
	We shall make use of Hildebrand's inequality $\Psi(a+b,y)-\Psi(a,y) \le \Psi(b,y)$, valid for $y \ge C$ and $a,b \ge y$. It implies
	\begin{equation}\label{eq:hildineq}
	\Psi(a+b,y)-\Psi(a,y) \le \Psi(b,y)+1
\end{equation}
for $y\ge C$ and all $a,b$. We apply \eqref{eq:hildineq} with $a=x-Cx/2^k$ and $b=2Cx/2^k$ to find that \eqref{eq:dyadicpsi} is bounded by
	\begin{equation}\label{eq:dyadicpsi2}
		\ll	\frac{1}{T}\sum_{1 \le k \le \log_2 T + 1} 2^k\left(\Psi\left(\frac{Cx}{2^k},y\right)+1\right) \ll \frac{1}{T}\sum_{1 \le k \le \log_2 T + 1} 2^k\left(\Psi\left(\frac{x}{2^k},y\right)+1\right)
	\end{equation}
	where in the second inequality we replaced $\Psi(Cx,y)$ with $\Psi(x,y)$ using \cite[Thm.~3]{HildebrandTenenbaum1986}.	To conclude, we recall Theorem 2.4 of \cite{LaBreteche2005} says $\Psi(x/d,y) \ll \Psi(x,y)/d^{\alpha}$ holds for $x \ge y \ge 2$ and $1 \le d \le x$. We apply this inequality with $d=2^k$ and obtain
	\begin{align}\label{eq:dyadicpsi22}
		\frac{1}{T}\sum_{1 \le k \le \log_2 T + 1} 2^k\left(\Psi\left(\frac{x}{2^k},y\right)+1\right) &\ll 1+\frac{\Psi(x,y)}{T}\sum_{1 \le k \le \log_2 T + 1} 2^{(1-\alpha)k}\ll 1+\frac{\Psi(x,y)\log T}{T^{\alpha}}
	\end{align}
	as needed.
\end{proof}
\subsection{Proof of Proposition \ref{prop:medu}}
We first truncate the Perron integral for $\Psi(x,y)$. We apply Lemma \ref{lem:psitrunc} with $\sigma = \betasad$ and $T=\exp ((\log y)^{4/3})$. The assumption $y \ge (\log x)^{1+\varepsilon}$ implies $\betasad \gg_{\varepsilon} 1$ and $\Psi(x,y) \ge x^{c_{\varepsilon}}$. Since $\alpha=\beta+O(1/\log y)$ \cite[Lem.~2]{HildebrandTenenbaum1986} it follows that $\alpha \gg_{\varepsilon} 1$ and so
\begin{equation} \Psi(x,y)=\frac{1}{2\pi i}\int_{\betasad-iT}^{\betasad+iT} \zeta(s,y)\frac{x^s}{s}ds+O_{\varepsilon}\left( \frac{x^{\betasad}\zeta(\betasad,y)+\Psi(x,y)}{T^{c_{\varepsilon}}}\right).
\end{equation}
We use Lemma \ref{lem:zetaratio} to bound the contribution of $1/\log y \le |\Im s| \le T$:
\begin{align} 
\int_{\betasad+i/\log y}^{\betasad+iT} \zeta(s,y)\frac{x^s}{s}ds &\ll x^{\betasad} \zeta(\betasad,y) \int_{1/\log y}^{T} \left|\frac{\zeta(\betasad+it,y)}{\zeta(\betasad,y)}\right|\frac{dt}{\betasad+t}\\
& \ll x^{\betasad} \zeta(\betasad,y) \int_{1/\log y}^{T} \exp\left(-\frac{c ut^2}{(1-\betasad)^2+t^2}\right)\frac{dt}{\betasad+t}\\
& \ll x^{\betasad} \zeta(\betasad,y) \left( \exp(-c u)\log T + \int_{1/\log y}^{\xi(u)/\log y} \exp\left(-\frac{c (\log x)( \log y)}{\log^2 (u+1)}t^2\right)dt \right)\\
& \ll x^{\betasad} \zeta(\betasad,y)\exp\left(-\frac{c u}{\log^2(u+1)}\right).
\end{align}
We estimate $x^{\betasad} \zeta(\betasad,y)$:
\begin{equation}\label{eq:xbetazetaest}
\begin{split}
x^{\betasad}\zeta(\betasad,y) &= \frac{x}{e^{u\xi(u)}}F(\betasad,y) G(\betasad,y)\\
&=\zeta(\beta)(\beta-1) \frac{x  e^{I(\xi)+\gamma}\log y}{e^{u\xi(u)}} G(\betasad,y) \ll_{\varepsilon} x\rho(u)\sqrt{(\log x)( \log y)} G(\betasad,y)
\end{split}
\end{equation}
using \eqref{eq:Fdef} and Lemma \ref{lem:rho size}. Finally, note that both $T$ and $\exp(u/\log^2(u+1))$ grow faster than any power of $\log x$. 
We turn to $\Lambda(x,y)$. We apply Corollary \ref{cor:lambdatrunc} with $\sigma=\betasad$ and 
\[ T = \frac{y^{1-\betasad}}{\log y} = \frac{e^{\xi(u)}}{\log y} \asymp \frac{u \log (u+1)}{\log y} \gg (\log \log y)^4.\]
We obtain
\[ \Lambda(x,y) = \frac{1}{2\pi i}\int_{\betasad-iT}^{\betasad+iT} F(s,y) \frac{x^s}{s} ds+ O_{\varepsilon}\left(  \frac{ux}{\exp(u\xi)}\right).\]
We now treat the range $1/\log y\le | \Im s| \le T$. By the definition of $F$,
\begin{equation}\label{eq:1toT}
\int_{\betasad+\frac{i}{\log y}}^{\betasad+iT} F(s,y) \frac{x^s}{s} ds\ll_{\varepsilon} \frac{x\log y}{\exp(u\xi)} \int_{1/\log y}^{T} |\zeta(\betasad+it)|  |\hat{\rho}(-\xi(u)+it \log y)| dt.
\end{equation}
First suppose $t \ge \pi/\log y$. By the second case of Lemma \ref{lem:i bounds}, this range contributes
\begin{equation}\label{eq:lambdatruncmiddle}
	\begin{split}
 &\ll_{\varepsilon} \frac{x \exp(I(\xi))\log y}{\exp(u\xi)} \exp\left( - \frac{u}{\pi^2 +\xi^2}\right) \int_{\pi/\log y}^{T} |\zeta(\betasad+it)|dt\\
& \ll x\rho(u) \sqrt{(\log x)(\log y)}\exp\left(- \frac{u}{\pi^2+\xi^2}\right)\int_{\pi/\log y}^{T} |\zeta(\betasad+it)|dt
	\end{split}
\end{equation}
using Lemma \ref{lem:rho size} in the second inequality.
Recall the second moment estimate for $\zeta$ given in Lemma \ref{lem:second}. It shows that right-hand side of \eqref{eq:lambdatruncmiddle} is bounded by
\[  \ll x\rho(u) \sqrt{(\log x )(\log y)}\exp\left(- \frac{u}{\pi^2+\xi^2}\right) T^{\max\{1,3/2-\betasad\}}\sqrt{\log T}\]
where we used the functional equation if $\betasad < 1/2$ (Lemma \ref{lem:functional}).
The contribution of $1/\log y \le t \le \pi/\log y$ to the right-hand side of \eqref{eq:1toT} is treated using the first part of Lemma \ref{lem:i bounds}, and we find that it is at most
\begin{equation}
\ll_{\varepsilon} \frac{x  \exp(I(\xi))\log y}{\exp(u\xi)} \int_{1/\log y}^{\pi/\log y} \exp\bigg(- \frac{(\log x)( \log y)}{2\pi^2}t^2\bigg)dt 
\ll_{\varepsilon}  x\rho(u) \exp(-cu),
\end{equation}
using Lemma \ref{lem:rho size} in the second inequality. In conclusion,
\[ \Lambda(x,y) = \frac{1}{2\pi i}\int_{\betasad-i/\log y}^{\betasad+i/\log y} F(s,y) \frac{x^s}{s} ds+ E\]
where
\[ E \ll_{\varepsilon} \frac{ux}{\exp(u\xi)} + x\rho(u) \left( \sqrt{(\log x)( \log y)}\exp\left(- \frac{u}{\pi^2+\xi^2}\right) T^{\max\{1,3/2-\betasad\}}\sqrt{\log T} + \exp(-cu)\right).\]
By our choice of $T$ and assumptions on $u$ and $y$, this can be absorbed in the error term of \eqref{eq:lambmed}.

\subsection{Proof of Proposition \ref{prop:small}}
We first truncate the Perron integral for $\Psi(x,y)$. We apply Lemma \ref{lem:psitrunc} with $\sigma=\betasad$ and our $T$, finding
\begin{equation} 
\Psi(x,y)=\frac{1}{2\pi i}\int_{\betasad-iT}^{\betasad+iT} \zeta(s,y)\frac{x^s}{s}ds+O\left( 1+ \frac{\Psi(x,y)\log T}{T^{\alpha}}+\frac{x^{\betasad}\zeta(\betasad,y)}{T}\right).
\end{equation}
In the considered range, $\Psi(x,y) \asymp x \rho(u)$. In particular, the error term $O(1)$ is acceptable since our $T$ is $\ll x\rho(u) \ll \Psi(x,y)$ and so $1 \ll \Psi(x,y) /T^{4/5}$. Additionally, $\beta \sim 1$ as $x \to \infty$ by Lemma \ref{lem:xilem} and $\alpha=\beta+O(1/\log y)$ \cite[Lem.~2]{HildebrandTenenbaum1986}, so $\alpha \sim 1$. This implies that $(\log T)/T^{\alpha} \ll 1/T^{4/5}$ and the error term $O(\Psi(x,y)(\log T)/T^{\alpha})$ is also acceptable. The estimate \eqref{eq:xbetazetaest} treats the last error term and finishes the estimation. We turn to $\Lambda(x,y)$. We apply Corollary \ref{cor:lambdatrunc} with our $T$, obtaining 
	\begin{equation}\label{eq:trunlambdasmall}
	\Lambda(x,y) = \frac{1}{2\pi i}\int_{\betasad-iT}^{\betasad+iT} F(s,y) \frac{x^s}{s} ds+ O\left(\log x +x  \exp(-u\xi) u \log (u+1) \frac{\log T}{T^{\sigma}}\right).
\end{equation}
In our range $x\rho(u) \asymp x^{1+o(1)}$, so the term $\log x$ is acceptable. We have $\exp(-u\xi) u\log (u+1) \ll \rho(u)$ by Lemma \ref{lem:rho size}, so the second term in the error term of \eqref{eq:trunlambdasmall} is also acceptable.
\section{Proofs of Theorems \ref{thm:psierror} and \ref{thm:psierror2}}\label{sec:rherror}
\begin{proposition}[Medium $u$]\label{prop:large}
Suppose $x \ge y \ge 2$ satisfy \[u \ge (\log y) ( \log \log y)^3.\]
Fix $\varepsilon >0$ and suppose $y \ge (\log x)^{1+\varepsilon}$ and $x \ge C_{\varepsilon}$. Let 
\[t_0:= (\log x)^{-1/3}(\log y)^{-2/3}, \qquad T := \exp(\min\{u/\log^2(u+1),(\log y)^{4/3}\}).\] Then $\Psi(x,y) = \Lambda(x,y)  G(\betasad,y) ( 1 + E)$ for
\begin{equation}
E \ll_{\varepsilon} \frac{ |G'(\betasad,y)|}{G(\betasad,y)\log x}  + \frac{ \max_{|v|\le t_0}|G''(\betasad+iv,y)|}{G(\betasad,y)(\log x)(\log y)}+\frac{\max_{|v|\le \frac{1}{\log y}} |G'(\betasad+iv,y)|\exp(-u^{1/3}/20)}{G(\betasad,y)\log x} + \frac{1}{T^{c_{\varepsilon}}}.
\end{equation}
\end{proposition}
\begin{proof}
Our strategy is to establish $\Psi(x,y) = \Lambda(x,y) G(\betasad,y)  (1 + E_1 + E_2) + E_3$ for
	\begin{align}
		E_1 &\ll_{\varepsilon} \frac{ |G'(\betasad,y)|}{G(\betasad,y)\log x} + \frac{ \max_{|v|\le t_0}|G''(\betasad+iv,y)|}{G(\betasad,y)(\log x) (\log y)},\\
		E_2 & \ll_{\varepsilon} \frac{\max_{|v|\le \frac{1}{\log y}} |G'(\betasad+iv,y)|\exp(-u^{1/3}/20)}{G(\betasad,y)\log x},\\
		E_3 &\ll_{\varepsilon} \frac{\Psi(x,y)+x\rho(u)G(\betasad,y)}{T^{c_{\varepsilon}}}.
	\end{align}
The theorem will then follow by rearranging, once we recall that $x\rho(u) \asymp_{\varepsilon} \Lambda(x,y)$. From Proposition \ref{prop:medu},
\begin{equation}\label{eq:diffintmed}
\begin{split}
\Psi(x,y) &-\Lambda(x,y)G(\betasad,y) \\
&=\frac{1}{2\pi i} \int_{\betasad-\frac{i}{\log y}}^{\betasad+\frac{i}{\log y}} ( G(s,y)-G(\betasad,y))F(s,y) \frac{x^s}{s}ds + O_{\varepsilon}\left( \frac{\Psi(x,y)+x\rho(u) G(\betasad,y)}{T^{c_{\varepsilon}}}  \right),
\end{split}
\end{equation}
which explains $E_3$.
Let $t_0$ be as in the statement of the proposition.
We upper bound the contribution of $t_0 \le |\Im s| \le 1/\log y$ to the integral in the right-hand side of \eqref{eq:diffintmed}.
We have
\[ |G(s,y)-G(\betasad,y)| \le |\Im s| \max_{ |t|\le |\Im s|} |G'(\betasad+it,y)|.\]
The triangle inequality shows, by definition of $F$, that
\begin{equation}
\int_{\betasad+it_0}^{\betasad+\frac{i}{\log y}} ( G(s,y)-G(\betasad,y))F(s,y) \frac{x^s}{s}ds \ll_{\varepsilon} \max_{|t| \le \frac{1}{\log y}} |G'(\betasad+it,y)| x^{\betasad} \log y \int_{t_0}^{\frac{1}{\log y}} t |e^{I(\xi-it\log y)}| dt.
\end{equation}
Since $-e^{-v^2/2}$ is the antiderivative of $e^{-v^2/2}v$, the first part of Lemma \ref{lem:i bounds} shows
\begin{align}
\int_{t_0}^{\frac{1}{\log y}} t |e^{ I(\xi-it\log y)}| dt &\ll \exp(I(\xi))\int_{t_0}^{\frac{1}{\log y}} t \exp(-(\log x) (\log y)t^2/(2\pi^2)) dt\\
&\ll \exp(I(\xi)) \frac{\exp(-u^{1/3}/(2\pi^2))}{(\log x)( \log y)}.
\end{align}
 Hence $t_0 \le |\Im s| \le 1/\log y$ contributes in total
\[ \ll_{\varepsilon} \max_{|t| \le 1/\log y}|G'(\betasad+it,y)| x\rho(u)\exp(-u^{1/3}/20)/\log x\]
where we used Lemma \ref{lem:rho size} to simplify.
Once we divide this by $\Lambda(x,y)G(\betasad,y) \asymp_{\varepsilon} x\rho(u)G(\betasad,y)$ we obtain the error term $E_2$. It remains to study the contribution of $|\Im s|\le t_0$ to the integral in the right-hand side of \eqref{eq:diffintmed}, which will yield $E_1$. We Taylor-expand the integrand at $s=\betasad$. We write $s=\betasad+it$, $|t|\le t_0$. We first simplify the integrand using the definition of $F$:
\begin{align}
\frac{F(s,y)x^{s}}{s} &= (\log y) K(s-1) e^{\gamma+I(\xi)}x^{\betasad+it} \exp(I(\xi-it \log y)-I(\xi))\\
&=  (\log y) K(s-1) x^{\betasad}e^{\gamma+I(\xi)} \exp(I(\xi-it \log y)-I(\xi)+it \log x).
\end{align}
We Taylor-expand $\log K(s-1)$ and $G(s,y)-G(\betasad,y)$:
\begin{align}
K(s-1) &= K(\betasad-1) (1+O_{\varepsilon}(t)),\\
G(s,y)-G(\betasad,y) &= itG'(\betasad,y)+ O(t^2 \max_{|v|\le t} |G''(\betasad+iv,y)|).
\end{align}
We expand $I(\xi-it \log y)-I(\xi)+it \log x$: 
\begin{equation}
	I(\xi-it \log y)-I(\xi)+it \log x = -\frac{t^2}{2}I''(\xi)\log^2 y + O( |t|^3 (\log x) (\log y)^2),
\end{equation}
where we used $I'(\xi(u))=u$ and $I^{(3)}(\xi(u)+it) \ll e^{\xi(u)}/(1+\xi(u)) \asymp u$. This implies
\begin{equation}\label{eq:taylorexpandularge}
	\exp(I(\xi-it \log y)-I(\xi)-it\log y) = \exp\left(-\frac{t^2}{2}I''(\xi)\log^2 y\right)( 1+ O( |t|^3 (\log x) (\log y)^2))
\end{equation}
for $|t| \le t_0$. By two basics properties of moments of the gaussian,
\begin{equation}\label{eq:gauss}
	\begin{split}
	\int_{-t_0}^{t_0} t\exp\left(-\frac{t^2}{2}I''(\xi)\log^2 y\right) dt &= 0,\\
	\int_{-t_0}^{t_0} |t|^{k}\exp\left(-\frac{t^2}{2}I''(\xi)\log^2 y\right) dt &\ll_k (I''(\xi)\log^2 y)^{-\frac{k+1}{2}} \ll_k ((\log x)(\log y))^{-\frac{k+1}{2}},
\end{split}	
\end{equation}
we find
\begin{equation}
\int_{\betasad-it_0}^{\betasad+it_0} ( G(s,y)-G(\betasad,y))F(s,y) \frac{x^s}{s}ds\ll_{\varepsilon} x^{\betasad}e^{I(\xi)} \left( \frac{|G'(\betasad,y)|\sqrt{\log y}}{(\log x)^{3/2} } +  \frac{\max_{|v|\le t_0}|G''(\betasad+iv)|}{(\log x)^{3/2}(\log y)^{1/2}}\right).
\end{equation}
By Lemma \ref{lem:rho size}, we can replace $x^{\betasad} e^{I(\xi)}$ with $x\rho(u)\sqrt{u}$, to obtain
\begin{equation}\label{eq:tay}
\int_{\betasad-it_0}^{\betasad+it_0} ( G(s,y)-G(\betasad,y))F(s,y) \frac{x^s}{s}ds\ll_{\varepsilon} x\rho(u)\left( \frac{|G'(\betasad,y)|}{\log x} +  \frac{\max_{|v|\le t_0}|G''(\betasad+iv)|}{(\log x)(\log y)}\right).
\end{equation}
Dividing by $G(\betasad,y)\Lambda(x,y) \asymp_{\varepsilon} G(\betasad,y) x\rho(u)$ gives the error term $E_1$.
\end{proof}
\begin{proposition}[Small $u$]\label{prop:smallu}
Suppose $x \ge y \ge C$ satisfy 
\begin{equation}\label{eq:smallurange}
	u \le (\log y) ( \log \log y)^3.
	\end{equation}
	 Let
	\begin{equation}
		t_0:= (\log x)^{-1/3}(\log y)^{-2/3}, \quad 
		t_1: = \frac{u\log(u+1)}{\log y},\quad	t_2 \in [(\log x)^{5}, y^{4/5}].
	\end{equation}
Then $	\Psi(x,y) = \Lambda(x,y) G(\betasad,y) ( 1 + E)$ for \begin{multline}
		E \ll \frac{ |G'( \betasad,y)|}{\log x} + \frac{ \max_{|v|\le t_0}|G''( \betasad+iv,y)|}{(\log x)(\log y)}+ \frac{\max_{|v|\le t_1} |G'( \betasad+iv,y)|\exp(-u^{1/3}/20)}{\log x}+t_2^{-4/5} \\+ \exp(-u/2)\left( \max_{|t|\le t_2}\left| \frac{G( \betasad+it,y)}{G( \betasad,y)}-1\right| +\left|\int_{t_1 \le |t| \le t_2} K( \betasad+it-1)x^{it} \frac{G( \betasad+it,y)-G( \betasad,y)}{G( \betasad,y)}\frac{dt}{t}\right|\right).
	\end{multline}
\end{proposition}
\begin{proof}
	Our strategy is to establish $\Psi(x,y) = \Lambda(x,y) G( \betasad,y)  (1 + E_1 + E_2 + E_3 + E_4) + E_5$ for
	\begin{equation}\label{eq:eis}
		\begin{split}
		E_1 &\ll \frac{ |G'( \betasad,y)|}{G( \betasad,y)\log x} + \frac{ \max_{|v|\le t_0}|G''( \betasad+iv,y)|}{G( \betasad,y)(\log x)( \log y)},\\
		E_2 & \ll \frac{\max_{|v|\le t_1} |G'( \betasad+iv,y)|\exp(-u^{1/3}/20)}{G( \betasad,y)\log x},\\
		E_3 & \ll  \frac{\exp(-u/2)}{\log y} \int_{t_1 \le |t| \le t_2} \left| \frac{G( \betasad+it)-G( \betasad,y)}{G( \betasad,y)}\right| \frac{\log(|t|+2)}{t^2} dt, \\
		E_4 &\ll  \exp(-u/2) \left|\int_{t_1 \le |t| \le t_2} K( \betasad+it-1)x^{it} \frac{G( \betasad+it,y)-G( \betasad,y)}{G( \betasad,y)}\frac{dt}{t}\right|,\\
		E_5 &\ll t_2^{-4/5} (\Psi(x,y)+x\rho(u)G( \betasad,y)).
	\end{split}
	\end{equation}
	The proposition will then follow by rearranging and the fact that $G(\beta,y) \asymp 1$ in the considered range, unconditionally, as follows from Corollary \ref{cor:logg1size} and Lemma \ref{lem:logGderivaccurate2}. From Proposition \ref{prop:small} with $T=t_2$,
	\begin{equation}\label{eq:diffintmed2}
		\begin{split}
			\Psi(x,y) -&\Lambda(x,y)G( \betasad,y) \\
			&=\frac{1}{2\pi i} \int_{ \betasad-it_2}^{ \betasad+it_2} ( G(s,y)-G( \betasad,y))F(s,y) \frac{x^s}{s}ds + O( t_2^{-4/5}(\Psi(x,y)+x\rho(u) G( \betasad,y)) ),
		\end{split}
	\end{equation}
	which explains $E_5$. For $|\Im s| \le t_0$, we Taylor-expand $I(\xi-it\log y)$ as in the medium $u$ range and obtain the contribution of $E_1$ (see \eqref{eq:tay})
	We treat the contribution of $|\Im s| \in [t_0,t_1]$.
	We replace $G(s,y)-G( \betasad,y)$ with
	\[ |G(s,y)-G( \betasad,y)| \le |\Im s| \max_{0 \le |t|\le |\Im s|} |G'( \betasad+it,y)|.\]
	The first two parts of Lemma \ref{lem:i bounds} show
	\begin{align}
	&\int_{ \betasad+it, \, |t| \in [t_0,t_1]} ( G(s,y)-G( \betasad,y))F(s,y) \frac{x^s}{s}ds  \\
		&\ll  \max_{|t|\le t_1} |G'( \betasad+it,y)|  x\rho(u)(\log y)\sqrt{u}\int_{|t| \in [t_0,t_1]} |t|\left( \exp\left(-\frac{t^2 (\log x)( \log y)}{2\pi^2}\right) + \exp(-u/(\pi^2+\xi^2))\right) dt\\
		&\ll  \max_{|t|\le t_1} |G'( \betasad+it,y)| x\rho(u)\sqrt{u}\frac{\exp(-u^{1/3}/2\pi^2)}{\log x}.
	\end{align}
	This explains $E_2$. It remains to consider $t_2 \ge |\Im s| \ge t_1$. We use the third part of Lemma \ref{lem:i bounds} to replace $\hat{\rho}((s-1)\log y)$, appearing in $F(s,y)$, with its approximation:
	\begin{multline}\label{eq:lastrange}
	 \int_{ \betasad+it, \, |t| \in [t_1,t_2]} ( G(s,y)-G( \betasad,y))F(s,y) \frac{x^s}{s}ds \\
		=(\log y) x^{ \betasad} \int_{s= \betasad+it, \, |t| \in [t_1,t_2]} K(s-1)x^{it} ( G(s,y)-G( \betasad,y))\left( \frac{i}{t\log y} + O\left( \frac{u\log (u+1)}{t^2 \log^2 y}\right)\right) ds.
	\end{multline}
	Recall $x^{ \betasad} \ll x\rho(u) \sqrt{u} \exp(-I(\xi(u)))$ by Lemma \ref{lem:rho size}, and that $I(\xi(u)) \sim u$ since a change of variables shows $I(r)= \mathrm{Li}(e^r)+O(\log r)\sim e^r/r$. The contribution of the error term in the right-hand side of \eqref{eq:lastrange} is
	\begin{align}
		&\ll x^{ \betasad} \log y \int_{s= \betasad+it, \, |t| \in [t_1,t_2]} |K(s-1)x^{it} ( G(s,y)-G( \betasad,y))| \frac{u\log (u+1)}{t^2 \log^2 y} |ds| \\
		& \ll \frac{x\rho(u)\exp(-2u/3)}{\log y}  \int_{|t|\in[t_1,t_2]}| G( \betasad+it)-G( \betasad,y)| \frac{|\zeta( \betasad+it)| | \betasad+it-1|} {t^2 | \betasad+it|} dt.
	\end{align}
	If $|t|\le 2$ we use $|\zeta( \betasad+it)( \betasad+it-1)| \ll 1$ while if $|t| \ge 2$ we use Lemma \ref{lem:convex}, to obtain an error term of size $E_3$. The main term of \eqref{eq:lastrange} gives $E_4$.
\end{proof}

\subsection{Proof of Theorem \ref{thm:psierror}: medium \texorpdfstring{$u$}{u}}\label{sec:psierror}
Here we prove Theorem \ref{thm:psierror} in the range  \eqref{eq:usmall}.
We obtain from Proposition \ref{prop:large} that \textit{unconditionally}
\begin{equation}\label{eq:psiuncon}
\Psi(x,y) = \Lambda(x,y) G(\betasad,y) ( 1+E)
\end{equation}
for 
\begin{equation}
	E \ll_{\varepsilon} \frac{\max_{|v|\le 1} |G'(\betasad+iv,y)|}{G(\betasad,y)\log x} + \frac{ \max_{|v|\le 1}|G''(\betasad+iv,y)|}{G(\betasad,y)(\log x )(\log y)}+ \frac{1}{y}.
\end{equation}
Because we assume $y \ge (\log x)^{2+\varepsilon}$, we have $\betasad\ge 1/2+ c_{\varepsilon}$. Under RH, $\log G(\betasad,y) = O_{\varepsilon}(1)$ by Lemma \ref{lem:sizeg}. To bound the quantities appearing in $E$, we write $G(\betasad+it,y)$ as $G_1(\betasad+it,y)$ times $G_2(\betasad+it,y)$. Lemma \ref{lem:logGderivaccurate2} and equation \eqref{eq:pos} tell us that
\begin{equation}\label{eq:logg2i}
	(\log G_2)^{(i)}(\betasad+it,y) \ll_{\varepsilon} (\log y)^{i-1} y^{\frac{1}{2}-\betasad} 
\end{equation}
for $i=0,1,2$ and $t \in \RR$.  Corollary \ref{cor:logg1size} says that under RH
\begin{align}\label{eq:logg1form}
	(\log G_1)^{(i)}(\betasad+it,y) &= (-1)^i(\log y)^{i-1}y^{-\betasad-it} ( \psi(y)-y+ O_{\varepsilon}(y^{\frac{1}{2}})) \ll_{\varepsilon} (\log y)^{i+1} y^{\frac{1}{2}-\beta}
\end{align}
for all $i=0,1,2$ and $|t|\le 1$. Putting these two together, one obtains \eqref{eq:firstpart}.

\subsection{Proof of Theorem \ref{thm:psierror}: small \texorpdfstring{$u$}{u}}
Here we prove Theorem \ref{thm:psierror} for $u$ in the range \eqref{eq:smallurange}. In this range, $\betasad=1+o(1)$ and $\Psi(x,y)=x^{1+o(1)}$. Moreover, $\log G(\betasad,y) = O(1)$ unconditionally by Corollary \ref{cor:logg1size} and Lemma \ref{lem:logGderivaccurate2}. The hardest range of the proof will be $u\asymp 1$. Before proceeding with the actual proof, note that from Proposition \ref{prop:smallu} and the triangle inequality, it follows that
\begin{equation}\label{eq:smallutriangle} \Psi(x,y) = \Lambda(x,y) G(\betasad,y) \left(1+O\left(t_2^{-4/5} + t_2 \max_{|t|\le t_2}|G'(\betasad+it,y)|+\max_{|t|\le 1}|G''(\betasad+it,y)|\right)\right)
\end{equation}
holds unconditionally for $t_2 \in [(\log x)^5, y^{4/5}]$ and the range $x \ge y \ge C$, $u \le (\log y)(\log \log y)^3$.

We obtain from Proposition \ref{prop:smallu} with $t_2=y^{4/5}$ that
\[ \Psi(x,y) = \Lambda(x,y) G(\betasad,y) ( 1+E_1 + E_2 + E_3 + E_4 + y^{-3/5})\]
for $E_i$ bounded in \eqref{eq:eis}.
We write $G(\betasad+it,y)$ as $G_1(\betasad+it,y)$ times $G_2(\betasad+it,y)$. 
By Lemma \ref{lem:logGderivaccurate2} and \eqref{eq:pos},
\begin{equation}\label{eq:logg2i2}
	(\log G_2)^{(i)}(\betasad+it,y) \ll (\log y)^{i-1}u\log (u+1) y^{-\frac{1}{2}}
\end{equation}
for $i=0,1,2$ and $t \in \RR$ where we simplified $y^{-\betasad}$ using \eqref{eq:ybetasize}. From now on we assume RH. Corollary \ref{cor:logg1size} implies
\begin{align}\label{eq:logg1form2}
	(\log G_1)^{(i)}(\betasad+it,y) &\ll \frac{ (\log y)^{i-1}u \log (u+1)}{y}(|\psi(y)-y|+y^{\frac{1}{2}})
\end{align}
for $i=0,1,2$ when $|t|\le 1$. 
As in the medium $u$ case, one can bound $E_1$ by an acceptable quantity using our estimates for $(\log G_1)^{(i)}$ and $(\log G_2)^{(i)}$.
Recall
\[E_2  \ll \frac{\max_{|v|\le t_1} |G'( \betasad+iv,y)|\exp(-u^{1/3}/20)}{G( \betasad,y)\log x}\]
where $t_1 = u\log (u+1)/\log y$. If $t_1 \le 1$ we bound $E_2$ in the same way we bounded $E_1$. Otherwise we use \eqref{eq:gen2}, which implies that
\begin{align}\label{eq:logg1form22}
	(\log G_1)^{(i)}(\betasad+it,y) &\ll  (\log y)^{i+1}u \log (u+1) y^{-\frac{1}{2}}
\end{align}
holds for $i=0,1,2$ and $|t|\le  y^{9/10}$. This shows that, if $t_1>1$, i.e.~$u\log (u+1)\ge \log y$,
\[ E_2 \ll \frac{(\log y)^2 u \log(u+1)\exp(-u^{1/3}/20)}{y^{\frac{1}{2}}\log x}  \ll \log (u+1) y^{-\frac{1}{2}}.\]
This is an acceptable contribution when $u\log (u+1)> \log y$.
We now study $E_3$ and $E_4$. Due to $G(\beta+it,y)/G(\beta,y)$ being very close to $1$ in our considered range by \eqref{eq:logg2i2} and \eqref{eq:logg1form22}, we may replace \[G(\betasad+it,y)/G(\betasad,y)-1\] by \[\log G(\betasad+it,y)-\log G(\betasad,y)\]
an incur a negligible error, in both $E_3$ and $E_4$. So to show $E_3$ is acceptable we need to prove
\begin{equation}\int_{t_1 \le |t| \le y^{4/5}} | \log G( \betasad+it,y)-\log G( \betasad,y)| \frac{\log(|t|+2)}{t^2} dt \ll \frac{e^{u/3}}{y} (|\psi(y)-y|+y^{\frac{1}{2}}).
\end{equation}
This is shown using the bound
\begin{equation}\label{eq:pt} \log G( \betasad+it,y) \ll \frac{u \log (u+1)}{y \log y} (|\psi(y)-y| + y^{\frac{1}{2}}\log^2(|t|+2)), \qquad |t| \le y^{9/10},
\end{equation}
which is a consequence of \eqref{eq:gen2} and \eqref{eq:logg2i2}. To handle $E_4$ it remains to prove
\begin{equation}\label{eq:need}
\int_{t_1 \le |t| \le y^{4/5}} K(\beta+it-1)x^{it}( \log G( \betasad+it,y)-\log G( \betasad,y)) \frac{dt}{t} \ll_{\varepsilon} \frac{e^{u/2}}{y \log y} (|\psi(y)-y|+y^{\frac{1}{2}}).
\end{equation}
Here we cannot use the triangle inequality and put absolute value inside the integral. Indeed, if we use the pointwise bound \eqref{eq:pt}, along with our bounds for $\zeta$ (Lemmas \ref{lem:convex} and \ref{lem:second}), we get a bound which falls short by a factor of $(\log y)^3$. We shall overcome this by several integrations by parts as we now describe. 

To deal with the contribution of $\log G(\beta,y)$ to \eqref{eq:need} we use \eqref{eq:pt} with $t=0$ along with the bound
\[ \int_{t_1 \le |t| \le y^{4/5}} K(\beta+it-1)x^{it}\frac{dt}{t} \ll u^2\]
which follows by integration by parts, where we replace $x^{it}$ by its antiderivative $x^{it}/\log x$. 

Note that due to integration by parts, derivatives of $\zeta$ arise. This means that in addition to Lemmas \ref{lem:convex} and \ref{lem:second} we need the bounds $\zeta^{(k)}(s) \ll_k (1+(|t|+4)^{1-\sigma})\log^{k+1}(|t|+4)$ and $\int_{1}^{T} |\zeta^{(k)}(\sigma+it)|^2 dt \ll_k T$ for $\sigma \in [2/3,1]$ and $T, |t| \ge 1$. These bounds follow from Lemmas \ref{lem:convex} and \ref{lem:second} through Cauchy's integral formula. 

To deal with the contribution of $\log G(\beta+it,y)$ to \eqref{eq:need} we write it $\log G_1( \betasad+it,y)+\log G_2( \betasad+it,y)$ and obtain two integrals which we bound separately. 
\subsubsection{Treatment of \texorpdfstring{$\log G_1$}{logG1}}
Recall we assume $y \le x^{1-\varepsilon}$. We want to show
\begin{equation}\label{eq:need2}
	\int_{t_1 \le |t| \le y^{4/5}} K(\beta+it-1)x^{it} \log G_1( \betasad+it,y) \frac{dt}{t} \ll_{\varepsilon} \frac{e^{u/2}}{y \log y} (|\psi(y)-y|+y^{\frac{1}{2}}).
\end{equation}
We integrate by parts, replacing $x^{it}$ by its antiderivative, reducing matters to showing
\begin{equation}\label{eq:need3}
 \frac{1}{\log x}\int_{t_1\le |t| \le y^{4/5}}
K(\beta+it-1) x^{it}  \frac{G_1'}{G_1}(\beta+it,y) \frac{dt}{t}  \ll_{\varepsilon} \frac{e^{u/2}}{y \log y} (|\psi(y)-y|+y^{\frac{1}{2}}).
\end{equation}
We divide and multiply the integrand by $y^{it}$, so the left-hand side of \eqref{eq:need2} is now
\begin{equation}\label{eq:need4}
\frac{1}{\log x}\int_{t_1\le |t| \le y^{4/5}}
	K(\beta+it-1) (x/y)^{it} H(t)\frac{dt}{t}
\end{equation}
where $H(t):= y^{it}(G'_1/G_1)(\beta+it,y)$. From Lemma \ref{lem:logg1formula},
\[ y^{\beta} \cdot H(t) =  \sum_{|\Im (\rho)-t| \le 2y^{4/5}}\frac{y^{\rho}}{\rho-\beta-it}+O(y^{\frac{2}{5}}) \ll |\psi(y)-y|+ y^{\frac{1}{2}}\log^2 (|t|+2)\]
and, for $k=1,2,3$,
\[  y^{\beta} \cdot H^{(k)}(t) = (k+1)!i^{k}   \sum_{|\Im (\rho)-t| \le 2y^{4/5}}\frac{y^{\rho}}{(\rho-\beta-it)^{k+1}}+O(y^{\frac{2}{5}})\ll  y^{\frac{1}{2}}\log (|t|+2). \]
We integrate by parts 3 times, replacing $(x/y)^{it}$ by its antiderivative. We are guaranteed to get enough saving since $\log(x/y) \gg_{\varepsilon} \log x$. 
\subsubsection{Treatment of \texorpdfstring{$\log G_2$}{logG2}}
The function $\log G_2(\beta+it,y)$ is given as a sum over proper primes powers. As the cubes and higher powers contribute at most $\ll   y^{-2/3+o(1)}$ to it by the prime number theorem (see \cite{gorodetsky2022smooth}), we can replace $\log G_2(\beta+it,y)$  with the prime sum $\sum_{y^{1/2}<p \le y} p^{-2(\beta+it)}/2$, so we are left to show 
\[\sum_{y^{1/2}<p \le y} p^{-2\beta}\int_{t_1 \le |t| \le y^{4/5}} K(\beta+it-1)(x/p^2)^{it} \frac{dt}{t} \ll \frac{e^{u/2}}{y^{\frac{1}{2}} \log y}.\]
For a given $p$, the pointwise bound $(x/p^2)^{it} \ll 1$ leads to the above integral being bounded by $\ll \log y$. This is good enough for the primes $p \in [y^{1/2}\log y, y]$, since 
\[\sum_{y^{\frac{1}{2}}\log y \le p\le y}p^{-2\beta} \log y \asymp \frac{u\log (u+1)}{y^{\frac{1}{2}}\log y}.\]
For the primes $p \in (y^{1/2},y^{1/2}\log y)$ we integrate by parts, replacing $(x/p^2)^{it}$ by its antiderivatives.
\subsection{Proof of Theorem \ref{thm:psierror2}}
Suppose $(\log x)^{3} \ge y \ge (\log x)^{4/3+\varepsilon}$. It follows from Proposition \ref{prop:large} that
 $\Psi(x,y) = \Lambda(x,y)  G(\betasad,y) ( 1 + E)$ holds unconditionally for
\begin{equation}\label{eq:Eupper}
	E \ll_{\varepsilon} \frac{ |G'(\betasad,y)|}{G(\betasad,y)\log x}   + \frac{ \max_{|v|\le t_0}|G''(\betasad+iv,y)|}{G(\betasad,y)(\log x)(\log y)} + \frac{\max_{|v|\le \frac{1}{\log y}} |G'(\betasad+iv,y)|}{G(\betasad,y)\exp(u^{1/3}/20)}  +\frac{1}{y}
\end{equation}
where $t_0$ is given in the proposition. It remains to bound the quantities appearing in $E$. From now on we assume RH. Let $A:= (\log x) / y^{1/2}$. We will prove the  stronger bound
\begin{align} 
	E \ll_{\varepsilon} \frac{|\psi(y)-y|+y^{\frac{1}{2}}}{y} \bigg( 1+ u \frac{|\psi(y)-y|+y^{\frac{1}{2}}}{y}\bigg)+ \frac{\max\{A,A^2\}}{u\max\{1,|\log A|\}} \bigg( 1+\frac{\max\{A,A^2\}}{\max\{1,|\log A|\}}\bigg),
\end{align}
which implies the theorem using $\psi(y)-y \ll y^{1/2} \log^2 y$. Recall we can always simplify $y^{-\betasad}$ using \eqref{eq:ybetasize} as $\asymp_{\varepsilon} (\log x)/y$. In particular, $y^{1/2-\beta} \asymp_{\varepsilon} A$.
Recall $G=G_1 G_2$. Lemma \ref{lem:logGderivaccurate2} and equation \eqref{eq:pos} tell us that
\begin{equation}\label{eq:logg2ia}
	(\log G_2)^{(i)}(\betasad+it,y) \ll (\log y)^{i} \frac{\max\{A, A^2\}}{\max\{1,|\log A|\}}
\end{equation}
for $i=0,1,2$ and $t \in \RR$. Corollary \ref{cor:logg1size} says that under RH
\begin{align}\label{eq:logg1forma}
	(\log G_1)^{(i)}(\betasad+it,y) & \ll (\log y)^{i-1}\frac{\log x}{y}   ( |\psi(y)-y| + y^{\frac{1}{2}})
\end{align}
for $i=0,1,2$ and $|t|\le 1$. 
Applying \eqref{eq:logg2ia} and \eqref{eq:logg1forma} with $i=1$ shows
\[ \frac{ |G'(\betasad,y)|}{G(\betasad,y)} \frac{1}{\log x} \ll \frac{|\psi(y)-y|+y^{\frac{1}{2}}}{y}+  \frac{\max\{A,A^2\}}{u\max\{1,|\log A|\}}\]
which treats the first quantity in \eqref{eq:Eupper}. We now consider the third term in \eqref{eq:Eupper}. Observe
\begin{equation} \label{eq:firstobs}
	\frac{ \max_{|v|\le 1/\log y}|G'(\betasad+iv,y)|}{G(\betasad,y)\exp(u^{1/3}/20)} \le \frac{ \max_{|v|\le 1/\log y}|G(\betasad+iv,y)|}{G(\betasad,y)\exp(u^{1/3}/20)} \cdot \max_{|v|\le 1 }|(\log G)'(\betasad+iv,y)|.
\end{equation}
From \eqref{eq:logg2ia} and \eqref{eq:logg1forma} we have 
\begin{equation}
\max_{|v|\le 1 }|(\log G)'(\betasad+iv,y)| \ll  (\log x)^4,
\end{equation}
say, and, by \eqref{eq:pos} and \eqref{eq:logg1forma},
\begin{equation}
 \frac{ \max_{|v|\le 1/\log y}|G(\betasad+iv,y)|}{G(\betasad,y)}\le \exp(C_{\varepsilon} (\log y)^2 (\log x) /y^{1/2}),
\end{equation}
so that \eqref{eq:firstobs} leads to
\[\frac{\max_{|v|\le 1/\log y} |G'(\betasad+iv,y)|}{G(\betasad,y)\exp(u^{1/3}/20)} \ll_{\varepsilon} \frac{\exp(C_{\varepsilon} (\log y)^2 (\log x) /y^{1/2})}{ \exp(u^{1/3}/40)} \ll_{\varepsilon} \frac{1}{y}.\]
It remains to bound the second term in \eqref{eq:Eupper}. Observe
\begin{multline}\label{eq:secondterm22}
\frac{ \max_{|v|\le t_0}|G''(\betasad+iv,y)|}{G(\betasad,y)(\log x) (\log y)} \le \frac{ \max_{|v|\le t_0}|G(\betasad+iv,y)|}{G(\betasad,y)(\log x) (\log y)} \\
	\cdot ( \max_{|v|\le 1}|(\log G)''(\betasad+iv,y)| + \max_{|v|\le 1}|(\log G)'(\betasad+iv,y)|^2).
\end{multline}
By \eqref{eq:pos} we can bound the fraction in the right-hand side of \eqref{eq:secondterm22} by $O_{\varepsilon}(1)$:
\begin{multline}
\frac{ \max_{|v|\le t_0}|G(\betasad+iv,y)|}{G(\betasad,y)}\le \frac{ \max_{|v|\le t_0}|G_1(\betasad+iv,y)|}{G_1(\betasad,y)}\\
\le \exp\big( \int_{-t_0}^{t_0} |G_1'/G_1|(\beta+iv,y) dv\big)\le \exp(C_{\varepsilon}t_0 (\log y)^2(\log x)/y^{1/2}) \ll_{\varepsilon} 1.
\end{multline}
The derivatives of $\log G$ in the right-hand side of \eqref{eq:secondterm22} are handled by \eqref{eq:logg2ia} and \eqref{eq:logg1forma}, giving
\begin{multline}\label{eq:secondterm}
 \max_{|v|\le 1}|(\log G)''(\betasad+iv,y)| + \max_{|v|\le 1}|(\log G)'(\betasad+iv,y)|^2 \\
 \ll  \frac{(\log y )(\log x)}{y}(|\psi(y)-y|+y^{\frac{1}{2}}|) + \frac{(\log x)^2}{y^2} (|\psi(y)-y|+y^{\frac{1}{2}})^2\\
 +(\log y)^{2} \left( \frac{\max\{A, A^2\}}{\max\{1,|\log A|\}} + \frac{\max\{A, A^2\}^2}{\max\{1,|\log A|\}^2}\right).
\end{multline}
Dividing this by $(\log x) (\log y)$ gives a bound for the second term in \eqref{eq:Eupper}.
\appendix
\section{Review of \texorpdfstring{$\Lambda(x,y)$}{Lambda(x,y)}}
\subsection{\texorpdfstring{$\lambda_y$}{lambda y} and its Laplace transform}
Saias \cite[Lem.~4(iii)]{Saias1989} proved that $\lambda_y(v) \ll \rho(v)v^3 + e^{2v}y^{-v}$ holds for $y \ge 2$, $v \ge 1$. 
The following is a weaker version of his result which suffices for us.
\begin{lem}[Saias]\label{lem:lambdaupper}
	If $u \ge \max\{C,y+1\}$ we have $\lambda_y(u) \ll  (C/y)^{u}$. 
\end{lem}
\begin{proof}
	The condition $u \ge \max\{C,y+1\}$ ensures $e^{\xi(u-1)} \ge y$:
	\[  e^{\xi(u-1)} \ge (u-1) \xi(u-1) \ge y\xi(u-1) \ge y.\]
	Integrating the definition of $\lambda_y$ by parts gives
	\begin{equation}\label{eq:intbyparts}
		\lambda_y(u)=\rho(u)+\int_{0}^{u-1}(-\rho'(u-v))\{y^v\}y^{-v}dv + O(y^{-u}).
	\end{equation}
	By \eqref{eq:intbyparts} and the definition of $\rho$ we have
	\begin{equation}\label{eq:weaksaias} \frac{\lambda_y(u)}{\rho(u)} = 1 - \int_{0}^{u-1} \frac{\rho'(u-v)}{\rho(u)} \frac{\{y^v\}}{y^v} dv+ O( y^{-u})=  \int_{0}^{u-1} \frac{\rho(u-v-1)}{(u-v)\rho(u)} \frac{\{y^v\}}{y^v} dv+ O(1).
	\end{equation}
One has $\rho(u-v)\ll \rho(u)e^{v \xi(u)}$ uniformly for $0 \le v \le u$ \cite[Cor.~2.4]{Hildebrand1993}. Hence the integral on the right-hand side of \eqref{eq:weaksaias} is
	\[ \ll  \frac{\rho(u-1)}{\rho(u)} \int_{0}^{u-1} \left(\frac{e^{\xi(u-1)}}{y}\right)^v dv \le \frac{\rho(u-1)}{\rho(u)} (u-1) \ll u e^{\xi(u)}\]
	which is $\ll u^2 \log (u+1)$ by Lemma \ref{lem:xilem}. Hence
	\[ \lambda_y(u) \ll \rho(u) u^2\log(u+1) \ll u^{3/2} \log(u+1) \exp(I(\xi(u)) e^{-u\xi(u)} \le u^{3/2} \log(u+1) \exp(I(\xi(u))y^{-u}\]
	using Lemma \ref{lem:rho size}. We have $I(\xi(u)) \ll u$. As $u^{3/2}\log(u+1)$ may be absorbed in $C^u$, we are done.
\end{proof}
By Lemma \ref{lem:lambdaupper}, the contribution of $v \ge \max\{C,y+1\}$ to \eqref{eq:laplacelambday} is
\[\int_{\max\{C,y+1\}}^{\infty}| e^{-sv}\lambda_y(v)| dv \ll \int_{\max\{C,y+1\}}^{\infty} ( e^{-\Re s}C/y)^v dv<\infty.\]
This establishes
\begin{cor}\label{cor:lambdaconv}
	Fix $\varepsilon>0$. If $y \ge C_{\varepsilon}$ then $\hat{\lambda}_y$ converges absolutely for $\Re s > -(\log y)/(1+\varepsilon)$.
\end{cor}
\subsection{Asymptotics of \texorpdfstring{$\Lambda$}{Lambda}}\label{sec:lambdaasymp}
We define $r \colon [1,\infty) \to \RR$ by $r(t):=-\rho'(t)/\rho(t)=\rho(t-1)/(t \rho(t))$.
\begin{lem}\cite[Eq.~(6.3)]{Fouvry1996}\label{lem:rhoshifted2}
	For $0 \le v \le u-1$ and $u \ge 1$ we have
	\[ \rho'(u-v)-\rho'(u)e^{vr(u)} \ll \frac{\rho(u)ve^{vr(u)}}{u}( 1+v\log(u+1)).\]
\end{lem}
\begin{lem}\cite[Lem.~3.7]{LaBretecheTenenbaum}\label{lem:rxirel}
	For $u \ge 1$ we have $r(u) = \xi(u)+O(1/u)$.
\end{lem}
\begin{proposition}\label{prop:labte}Fix $\varepsilon>0$. Suppose $x\ge C_{\varepsilon}$. For $x \ge y \ge (\log x)^{1+\varepsilon}$,
	\[ \Lambda(x,y) =x \rho(u) K\left(- \frac{r(u)}{\log y}\right) \left(1+O_{\varepsilon}\left( \frac{1}{(\log x)(\log y)} + \frac{y}{x\log x}\right)\right).\]
\end{proposition}
Equation \eqref{eq:LaBTT} follows from Proposition \ref{prop:labte} using Lemma \ref{lem:rxirel}. Proposition \ref{prop:labte}, in slightly weaker form, is implicit in \cite[pp.~176--177]{LaBreteche2005}, and the proof given below follows these pages.
\begin{proof}
	For $u=1$ the claim is trivial since $\Lambda(x,x)=\lfloor x\rfloor$ \cite[Eq.~(3.2)]{debruijn1951}, so we assume $u>1$. Recall the integral representation $\zeta(s)=s/(s-1) -s \int_{1}^{\infty}\{t\}dt/t^{1+s}$ for $\Re s >0$ \cite[Eq.~(1.24)]{MV}. We apply it with $s=1-r(u)/\log y$ and perform the change of variable $t=y^v$ to obtain
	\begin{equation}\label{eq:rhoz}
		K(-r(u)/\log y) =1+r(u) \int_{0}^{\infty} e^{r(u)v}\{y^v\}y^{-v} dv.
	\end{equation}
	From \eqref{eq:rhoz} and \eqref{eq:intbyparts} we deduce  \begin{equation}\label{eq:intbyparts2}
		x\rho(u)K(- r(u)/\log y)-\Lambda(x,y)= x\int_{0}^{\infty}(\rho'(u-v)-\rho'(u)e^{r(u)v})\{y^v\}y^{-v}dv+ O(1).
	\end{equation}
	It remains to show that the right-hand side of \eqref{eq:intbyparts2} is \[\ll_{\varepsilon} x\rho(u)\left(\frac{1}{(\log x)(\log y)} +\frac{y}{x \log x}\right).\] 
	It is convenient to set
	\begin{equation}\label{eq:a}
		a :=\log\left(\frac{y}{e^{r(u)}}\right)= (\log y) - r(u) \ge \frac{\varepsilon}{2}\log y,
	\end{equation}
	where the inequality is due to Lemmas \ref{lem:rxirel} and \ref{lem:xilem} and our assumptions on $x$ and $y$.
	By Lemma \ref{lem:rhoshifted2}, the contribution of $0 \le v \le u-1$ to the right-hand side of \eqref{eq:intbyparts2} is
	\begin{align}
		&\ll \frac{x \rho(u)}{u} \int_{0}^{u-1} \left( \frac{e^{r(u)}}{y}\right)^{v} v(1+v\log(u+1))dv\\
		&= \frac{x \rho(u)}{u}\left( -e^{-av} \left( \frac{\log(u+1)}{a}v^2+\frac{2\log(u+1)+a}{a^2}v + \frac{2\log(u+1)+a}{a^3}\right)\right) \Big|^{v=u-1}_{v=0}.
	\end{align}
	Using $e^{(u-1)a}\gg \max\{(u-1)a, (u-1)^2 a^2\}$ and \eqref{eq:a} we find that the last quantity is $\ll_{\varepsilon} x\rho(u)/((\log x)(\log y))$ which is acceptable. For $v > u-1$, $\rho'(u-v)=0$ and that part of the integral (times $x$) is estimated as
	\[\ll x(-\rho'(u)) \int_{u-1}^{\infty}e^{-av}dv =x\rho(u) r(u)\frac{e^{-a(u-1)}}{a}\ll_{\varepsilon} x\rho(u) \log(u+1) \frac{e^{-a(u-1)}}{\log y}.\]
	If $u \ge 2$ this is $\ll_{\varepsilon} x\rho(u) /((\log x)(\log y))$, otherwise this is $\ll x\rho(u)(y/x)/\log x$. Both cases give an acceptable contribution.
\end{proof}
 \subsection*{Acknowledgements}
We are grateful to Sacha Mangerel for asking us about the integer analogue of \cite{gorodetsky}. We thank the referee for useful suggestions and comments that improved the manuscript. This project has received funding from the European Research Council (ERC) under the European Union's Horizon 2020 research and innovation programme (grant agreement No 851318).
\bibliographystyle{abbrv}
\bibliography{references}

\Addresses
\end{document}